\documentclass{article}  
\usepackage{graphicx}
\usepackage{amsmath}
\usepackage{amsthm}
\usepackage[T1]{fontenc}
\usepackage{amssymb}
\usepackage[utf8]{inputenc}
\usepackage{graphicx}
\usepackage{indentfirst}
\usepackage[hyphens]{url} 
\usepackage[english]{babel}	
\usepackage{hyperref}			
\usepackage{color}
\usepackage{float} % to center tables
\usepackage{tikz}
\usetikzlibrary{arrows}
\usepackage{caption}
\usepackage{subcaption}
\usepackage{babel} 
\usepackage{wrapfig}
\usepackage{envmath}
\numberwithin{equation}{section}
\usepackage[justification=centering]{caption} % to center the caption
\usepackage{pgfplots} % to plot images

\addto\captionsfrench{}

\def\Hk1{\mathrm{H}^{k+1}}
\def\HH1{\mathrm{H}^1}
\def\Hexpo{\mathrm{H}^}
\def\L2{\mathrm{L}^2}  
\def\LL{\mathrm{L}^}  
\def\Linf{\mathrm{L}^\infty} 
\def\Hdeux2{\mathrm{H}^2}

\def\c{\mathcal{C}^} 
\def\P{\mathbb{P}^}
\def\omgam{(\Omega,\Gamma)}

\def\d{\mathrm{d}}

\def\nn{\boldsymbol{\mathrm{n}}}

\def\na{\nabla}
\def\div{\mathrm{div}}
\def\dist{\mathrm{dist}}
\def\diff{\mathrm{D}}
\def\ds{\mathrm{d}s}
\def\dx{\mathrm{d}x}

\def\omhh{\Omega_h}
\def\omh1{\Omega_h^{(1)}}
\def\omhr{\Omega_h}

\def\ghh{\Gamma_h}
\def\ghr{\Gamma_h}

\def\ft{F_T}
\def\fte{F_T^{(e)}}
\def\ftr{F_T^{(r)}}
\def\ftre{F_{T^{(r)}}^{(e)}}
\def\ftdeux{\ft^{(2)}}
\def\Ghr{G_h^{(r)}}
\def\Ghdeux{{G}_h^{(2)}}

\def\ahell{a_h^\ell}

\def\Jb{J_b}
\def\Jh{J_h}
\def\Jblifte{J_b^\ell}
\def\Jhlifte{J_h^\ell}

\def\tref{\hat{T}}
\def\trefminissigma{\tref\backslash\hat{\sigma}}
\def\hatsigma{\hat{\sigma}}

\def\te{{T}^{(e)}}

\def\tdeux{{T}^{(2)}}
\def\tr{{T}^{(r)}}

\def\tauh{\mathcal{T}_h^{(1)}}
\def\taudeux{\mathcal{T}_h^{(2)}}
\def\taur{\mathcal{T}_h^{(r)}}
\def\taue{\mathcal{T}_h^{(e)}}

\def\lambdaetoile{\lambda^*}

\def\hatx{\hat{x}}
\def\haty{\hat{y}}
\def\hatv{\hat{v}}

\def\Aomega{\mathbf{A}}%_{\Omega}}

\def\eu{e(\u)}

\def\v{\boldsymbol{v}}
\def\w{\boldsymbol{w}}
\def\ev{e(\v)}

\def\Tr{\mathrm{Tr}}

\def\H{\Hexpovect{1}(\Omega)} % a changer avec Ventcel
\def\Ihlifted{\boldsymbol{I}^\ell}
\def\transGhr{\mathcal{G}}%{(\G){\transpose}}
\def\Gh{\mathcal{G}{\transpose}}

\def\muomega{\mu}
\def\lambdaomega{\lambda}

\def\euh{e(\uh)}

\def\f{\boldsymbol{f}}
\def\g{\boldsymbol{g}}
\def\u{\boldsymbol{u}}
\def\e{\mathbf{e}}
\def\Hvect1{\mathbf{\HH1}}
\def\Hexpovect{\mathbf{{H}}^}
\def\Lvect2{\mathbf{{L}}^2}

\def\uh{\boldsymbol{u}_h}
\def\vh{\boldsymbol{v}_h}
\def\Vhd{\boldsymbol{V}_h}

\def\gh{\boldsymbol{g}_h}

\def\Vhlifted{\Vhd^\ell}
\def\uhlifte{\uh^\ell}
\def\vhlifte{\vh^\ell}

\def\eL2{e_{\Lvect2(\Omega)}}
\def\eH1{e_{\mathbf{H}^1_0(\Omega)}}
\def\eLGamma{e_{\Lvect2(\Gamma)}}

\newcommand{\fonction}[5]{\begin{array}[t]{lrcl}#1 :&#2 &\longrightarrow &#3\\&#4& \longmapsto &#5 \end{array}}

\newtheorem{theorem}{Theorem}
\numberwithin{theorem}{section}

\newtheorem{lem}[theorem]{Lemma}

\newtheorem{corollary}[theorem]{Corollary}
\newtheorem{proposition}[theorem]{Proposition}
\newtheorem{definition}[theorem]{Definition}
\newtheorem{remarque}[theorem]{Remark}
\newtheorem{remark}[theorem]{Remark}
\newtheorem{ex}[theorem]{Example}

\def\R{\mathbb{R}}      
  
\def\I{\mathrm{I_d}} % matrice identite
\def\transpose{^\mathsf{T}}

\def\Ih1{\mathcal{I}^{(1)}}

\def\Ihlifte{\mathcal{I}^\ell}

\newcommand{\blue}[2][black]{{\textcolor{#1}{#2}}} % change blue to black when the manuscript is finished revision

\def\restriction#1#2{\mathchoice
              {\setbox1\hbox{${\displaystyle #1}_{\scriptstyle #2}$}
              \restrictionaux{#1}{#2}}
              {\setbox1\hbox{${\textstyle #1}_{\scriptstyle #2}$}
              \restrictionaux{#1}{#2}}
              {\setbox1\hbox{${\scriptstyle #1}_{\scriptscriptstyle #2}$}
              \restrictionaux{#1}{#2}}
              {\setbox1\hbox{${\scriptscriptstyle #1}_{\scriptscriptstyle #2}$}
              \restrictionaux{#1}{#2}}}
\def\restrictionaux#1#2{{#1\,\smash{\vrule height .8\ht1 depth .85\dp1}}_{\,#2}}
\title{Using curved meshes to derive a priori error estimates for a linear elasticity problem with robin boundary conditions}
\author{Joyce Ghantous\footnote{IMB, UMR 5251, Univ. Bordeaux; 33400, Talence, France. Inria Bordeaux
Sud-Ouest, Team MEMPHIS; 33400, Talence, France, \texttt{\url{joyce.ghantous@inria.fr}}.} }

\begin{document}

\maketitle

\begin{abstract}
This work concerns the numerical analysis of the linear elasticity problem with a Robin boundary condition on a smooth domain. A finite element discretization is presented using high-order curved meshes in order to accurately discretize the physical domain. The primary objective is to conduct a detailed error analysis for the elasticity problem using the vector lift operator, which maps vector-valued functions from the mesh domain to the physical domain. Error estimates are established, both in terms of the finite element approximation error and the geometric error, respectively associated to the finite element degree and to the mesh order. These theoretical a priori error estimates are validated by numerical experiments in $2$D and~$3$D.
% This work investigates a linear elasticity module with Robin's boundary condition.~A variational formulation of this problem is presented, leading to a finite element discretization. To ensure the boundary condition is well-defined, the domain is assumed to be smooth, and thus high-order curved meshes are introduced to discretize the physical domain accurately. The primary objective is to conduct a detailed error analysis for the elasticity problem using the vector lift operator, which maps vector-valued functions from the mesh domain to the physical domain. Error estimates are established, both in terms of the finite element approximation error and the geometric error, respectively associated to the finite element degree $k\ge 1$ and to the mesh order~$r\ge 1$. To validate the theoretical a priori error estimates, numerical experiments are performed in $2$D and in $3$D.
\end{abstract}


\medskip


\section{Introduction}
\paragraph{Motivation.} 
This work is part of a broader research initiative focusing on the study of vibration properties of mechanical parts subjected to intense and variable rotational regimes. Specifically, it is interested in these parts' vibration properties when they are surrounded by thin surface layers-resulting from corrosion or specialized industrial treatments. The ultimate objective is to improve the understanding of these mechanical structures and optimize their design using shape optimization techniques. A critical first step towards achieving this goal is performing a detailed error analysis of the problem’s solution.
\medskip

This paper serves as an intermediate yet essential step towards the numerical analysis of an eigenvalue problem for elastic structures coated with a very thin layer of constant thickness. The domain and solution of the considered problem can be approximated {using an asymptotic expansion}: the thin layer is modeled by adapted boundary conditions (see e.g. \cite{Gvial,vial-these,haddar-these,ref-ventcel}). In other words, the approximated domain is not surrounded by a thin layer, but it is equipped with second order boundary conditions such as {Ventcel conditions}, also known as {generalized Robin conditions} (see \cite{Ventcel-1956,Ventcel-1959}). \medskip

\blue{As an initial step toward the goal of deriving error estimates for the linear elasticity problem with Ventcel boundary conditions, we consider in this work the case involving Robin boundary conditions. The Robin conditions in linear elasticity provides a versatile framework for modeling complex boundary interactions. In the context of the vibration of structures (e.g., turbine blades or aircraft wings), these conditions are commonly employed to represent interfaces with damping mechanisms or elastic (spring-like) supports. They capture energy dissipation and mechanical constraints at the boundary. There exist many studies in the literature on a priori error estimates for linear elasticity problems with various boundary conditions, such as Dirichlet, Neumann, and mixed formulations; see, for example, \cite{gatica2006priori, mora2020apriori, stewart1998tutorial, rannacher1980finite}. Additionally, numerous works have addressed a posteriori error estimates for linear elasticity problems under these boundary conditions; see, for instance, \cite{barrios2011posteriori,carstensen1998posteriori}. }

\paragraph{State of the art and main results.} In this context, we aim to consider non-polygonal domains and more precisely smooth domains. It will permit to consider
practical cases (with high order Ventcel boundary conditions) at a later date, while also ensuring the regularity of the solution to the problem under consideration. Consequently, under suitable hypotheses, the elasticity-Robin problem admits a unique regular solution (see \cite{ciarlet-elasticity-volume-3,duvant2012inequalities}). However, a problem arises here where the physical domain $\Omega$ and the mesh domain $\omhh$ differ. Indeed, any mesh of the domain will not exactly fit $\Omega$ and there will be a gap between $\Omega$ and $\omhh$, called the \textit{geometric error}. As discussed and rigorously studied in \cite{elliott,art-joyce-1,D1,art-joyce-2}, taking higher order meshes can help lessen this geometric error. Therefore, meshes of order $r$ (\textit{i.e.} with elements of polynomial degree $r$) will be considered to improve the asymptotic behavior of the geometric error with respect to the mesh size $h$. %We note that the domain of the mesh of order $r$, denoted~$\omhr$, does not fit the physical domain~$\Omega$. 
\medskip

The main objective of this work is  to establish an a priori error estimate related to the linear elasticity problem with a Robin boundary condition, on meshes of order $r\ge1$, using the $\P k$  Lagrange conformal finite element method, with $k\ge1$. To estimate the error between the discrete solution and the exact one, defined on different domains, we use a \textit{lift operator} to map the discrete solution onto the physical domain. Throughout the years, many scalar lift operators are defined in many works, like in~\cite{elliott,nedelec,dubois,Lenoir1986,Bernardi1989}.  \blue{The use of a lift has so far been restricted to scalar problems; no application to vector-valued settings, such as elasticity, has been reported in the literature. For instance, in \cite{elliott,ed}, a scalar lift is defined and employed to handle bulk-surface coupled problems. In \cite{Dz88,D1,D2}, surface lifts are defined in a scalar case of Laplace–Beltrami problems defined on a surface. In the latter case, the definition of the surface lift is quite straight forward, relying solely on a composition with the orthogonal projection from the discrete surface onto the continuous one.}

\medskip

\blue{In the present work, in order to proceed with the error analysis, a vectorial lift operator is defined generalizing the scalar {lift} given in~\cite{art-joyce-1}. Thus, we can lift vector-valued function from the discrete to the physical domain using this \textit{vectorial lift} defined in the upcoming section. Additionally, this operator preserves the essential properties of its scalar version, in particular the trace property, which is crucial for lifting adequately the discrete weak formulation of the considered problem. }

\medskip

Hence, we investigate the dependency of the computed errors with respect to the mesh size~$h$, the order of the finite element method $k$ and the order of the considered mesh $r$. We proceed with a non-isoparametric approach, \textit{i.e.} when taking distinct orders $k$ and $r$. {To the best of our knowledge, such an approach has not been explored in the context of elasticity problems.} It was considered in \cite{D1,D2} for the error analysis of the Laplace-Beltrami problem and in \cite{D3} for treating the spectral Laplace-Beltrami problem. Recently, in \cite{art-joyce-1,Jaca}, it was also used to estimate the error of the Poisson-Ventcel problem and in \cite{art-joyce-2} for estimating the error of a spectral diffusion problem with Ventcel boundary conditions. For completeness, we highlight the following papers \cite{ed,elliott} as examples of applications of the isoparametric approach, where $k=r$, to estimate a priori errors. %\blue{The incorporation of a linear elasticity framework introduces technical challenges arising from the use of vector-valued function spaces. Furthermore, the elasticity-specific surface terms require careful treatment throughout the analysis in order to rigorously establish error estimates with respect to the relevant parameters.} 
\medskip

\blue{Here, a linear elasticity problem with Robin conditions is studied leading to a weak formulation with elasticity-related integrals. This problem is discretized using the finite element method and subsequently lifted using the vectorial lift operator producing intricate elasticity-related discrete lifted integrals. This added complexity has a direct impact on the derivation of the geometric error estimators. An additional aspect that requires careful treatment is the use of appropriate vector-valued norms. Hence, error estimates are derived with respect to the relevant parameters in the~$\Hexpovect{1}(\Omega)$ norm, as well as in the $\Lvect2(\Omega)$ and $\Lvect2(\Gamma)$ norms.}\medskip

\blue{To support the theoretical results, numerical simulations are conducted to validate the theoretical a priori error estimates, on both $2$D and~$3$D domains. Two noteworthy phenomena are observed. Firstly, the errors on the quadratic meshes are better than expected: a gain of one convergence order with respect to the theory is depicted. This aligns with previous findings in the study of scalar problems \cite{art-joyce-1,art-joyce-2,D2}. What is particularly interesting is that this behavior extends to vector-valued problems. Secondly, a loss in convergence rates is observed on cubic meshes, in line with previous findings in~\cite{art-joyce-1,art-joyce-2}. The novelty of the present analysis lies in the extension of this phenomenon to the surface $\Lvect2$ error, which was not reported in the scalar case. This aspect is further examined in a dedicated section.}

\paragraph{General notations.}
Firstly, let us introduce the notations that we adopt in this paper. Throughout this paper, $\Omega$ is a nonempty bounded connected open subset of $\R^{d}$ $(d=2,3)$ with a smooth (at least $\c2$) boundary~$\Gamma:=\partial{\Omega}$. The unit normal to~$\Gamma$ pointing outwards is denoted by~$\nn$. We denote respectively by $\LL 2(\Omega)$ and  $\LL 2(\Gamma)$ the usual Lebesgue spaces endowed with their standard norms on $\Omega$ and $\Gamma$. Moreover, for any integer $p \geq 0$, $\Hexpo{p}(\Omega)$ denotes the usual Sobolev space endowed with its standard norm. We also consider the Sobolev spaces $\Hexpo{p}(\Gamma)$ on the boundary as defined e.g.~in \cite[\S 2.3]{ventcel1}. In the following, spaces of vector functions will be denoted by boldface letters. Thus, we denote~$\Lvect2(\Omega):=[\L2(\Omega)]^d$ and $\Lvect2(\Gamma):=[\L2(\Gamma)]^d$. Similarly, for any~$p \ge 0$, we have~$\Hexpovect{p}(\Omega):=[\Hexpo{p}(\Omega)]^d$ and ~$\Hexpovect{p}(\Gamma):=[\Hexpo{p}(\Gamma)]^d$. We denote by $\I$ the~$d\times d$ identity matrix and by $\P p$ the set of polynomials in~$\R^d$ of order $p$ or less. For two square real valued matrices $A$ and~$B$ of same size~$d\times d$,~$A:B$ denotes the term by term product~$A:B = \Tr(A^{\transpose}B) = \sum_{1 \le i, j\le d} \, a_{ij} \, b_{ij}$, also known as the Frobenius inner product, where $\Tr$ denotes the matrix trace. %The symbol $\otimes$ represents the tensor product: for two vectors~$a, b \in \R^d$,~$a \otimes b = ab{\transpose}$ (vectors here are in column form). 
Lastly, for any smooth vector field~$\u$,~$\na \u$ is the matrix whose~$i^{{\rm  th}}$ row is the gradient of the $i^{{\rm  th}}$ component of~$\u$. For any smooth matrix function~$S$ defined on $\Omega$ with values in $\R^{d \times d}$, with rows~$S_j$ for~$j=1, \dots, d$, the divergence of $S$ is $\div S : \Omega \to \R^d$ given by $\div(S)_j=\div(S_j)$, for~$j=1, \dots, d$.

\paragraph{Paper organization.} \blue{Section~\ref{sec:steps} outlines the main steps of this work, beginning with the definition of the system and its weak formulation. Section~\ref{sec:mesh-lift} introduces the high-order curved meshes and the vector lift operator, which form the key ingredients of the error analysis (see Appendices~\ref{mesh:appendix} and~\ref{appendix:lift} for further details). In Section~\ref{sec:FEM}, we define a Lagrangian finite element space and present the discrete formulation of the linear elasticity problem, together with their lifted versions on the physical domain~$\Omega$. The main theoretical result, providing the a priori error estimates, is then stated. Section~\ref{sec:numerical-ex-elas} is dedicated to numerical simulations on both $2$D and $3$D domains, illustrating the influence of the geometric order~$r$ and the finite element degree~$k$ on the convergence rates. Finally, Section~\ref{sec:error-estimation-proof} presents the detailed proof of the error estimates.}

\section{The main stepping stones}
\label{sec:steps}
\subsection{The linear elasticity problem with a Robin boundary condition}
We recall that  $\Omega$ is a nonempty bounded connected domain in $\R^{d}$, $d=2, 3$, with a smooth boundary~$\Gamma := \partial\Omega$. 
Here, we assume that $\Omega$ is an elastic body and we consider an isotropic elastic medium with {\it Lamé coefficients} $\muomega > 0$ and $\lambdaomega >0$, which are considered as constants for more simplicity in this paper (possible extension to variable coefficients will be discussed later on see Remark \ref{rem:lame-coeff-cst}). We define its associated {\it elastic or Hooke tensor} $\Aomega$ by,
\begin{equation}
    \label{eq:Hooke-tensor}
    \Aomega\xi := 2 \muomega \xi + \lambdaomega \Tr(\xi) \I,
\end{equation}
for all symmetric matrices $\xi\in \R^d\times\R^d$. We refer to \cite{elasticity-1,ciarlet-elasticity-volume-3} for more details.
% \medskip

Next, we define the {\it strain tensor} for any vector field $\u =(u_i)_{i=1,...,d} \in \Hexpovect{1}(\Omega)$ by, 
\begin{equation}
\label{eq:e(u)}
    \eu := \frac{1}{2}(\na \u + (\na \u){\transpose}),
\end{equation}
which is the symmetric part of the Jacobian matrix $\na \u$. 
\medskip

Considering sufficiently regular source terms $\f$ and $\g$, the elasticity problem that we will focus on is the following:
\begin{equation}
\label{sys:elast-Robin}
\arraycolsep=2pt
\{
\begin{array}{rcll}
-\div(\Aomega \eu) &=& \f & \text{in } \Omega, \\
\u + \Aomega \eu \, \nn &=& \g & \text{on } \Gamma.
\end{array}
\right.
\end{equation}
% where $\nn$ denotes the external unit normal to $\Gamma$. %The problem~\eqref{sys:elast-Robin} admits a unique solution, denoted $\u \in \Hexpovect{1}(\Omega)$, which we aim to approximate. \medskip

The variational formulation of Problem~\eqref{sys:elast-Robin} is obtained, using the integration by parts formula, and it is given by,
\begin{equation}
\label{fv_faible}
  \left\lbrace
      \begin{array}{l}
       \mbox{find }  \u \in \Hexpovect{1}(\Omega), \ \mbox{ such that,}  \\
       a(\u,\v) = l(\v), \,  \forall \ \v \in \Hexpovect{1}(\Omega),
      \end{array}
    \right.
\end{equation}
where the bilinear form $a$, defined on $\Hexpovect{1}(\Omega) \times \Hexpovect{1}(\Omega)$, is given by,
\begin{align*}
    a(\u,\v) & :=  \int_{\Omega} \Aomega (\eu) : \ev \, \dx + \int_{\Gamma} \u\cdot \v \, \ds  =  \int_{\Omega} \Aomega (\eu) : \na \v \, \dx  + \int_{\Gamma} \u\cdot \v \, \ds,
\end{align*} 
where the latter equation is a consequence of the symmetry property of the strain tensor in \eqref{rk:Aomega-sym} as detailed in Section \ref{sec:error-estimation-proof}. 
Additionally, the linear form $l$, defined on $\Hexpovect{1}(\Omega)$, is given by,
\begin{equation}
\label{eq:lin-form-elast}
     l(\v) := \int_{\Omega} \f\cdot \v \,\dx + \int_{\Gamma} \g\cdot \v \, \ds.
\end{equation}

The following theorem asserts the well-posedness of Problem \eqref{fv_faible}.This result is a direct consequence of the Lax-Milgram theorem, relying on the continuity of the symmetric bilinear form $a$ and its coercivity in the space $\Hexpovect1(\Omega)$, the latter being ensured by Korn's inequality (see \cite[Th. 1.8.2]{ciarlet-elasticity-volume-3} and \cite[\S 3.4]{duvant2012inequalities}). 

\begin{theorem}
\label{th_existance_unicite_u}
Let $\Omega$ and  $\Gamma= \partial \Omega$ be as stated previously. 
Let $\f \in \Lvect2(\Omega)$, $\g \in \Lvect2(\Gamma)$. Then there exists a unique solution 
$\u \in \Hexpovect{1}(\Omega)$ to Problem (\ref{fv_faible}). Additionally, there exists $c>0$ such that the following inequality holds,
$$
    \|\u\|_{\Hexpovect{1}(\Omega)} \le c ( \|\f\|_{\Lvect2(\Omega)} + \|\g\|_{\Lvect2(\Gamma)}).
$$
\end{theorem}

\subsection{Curved mesh and lift operator definitions}
\label{sec:mesh-lift}

Throughout this section, we briefly recall the definition of curved meshes of geometrical order~$r\ge 1$ of the domain~$\Omega$ and give the main associated notations. We refer to \cite[\S 3]{art-joyce-1} for details and rigorous definitions (in particular concerning the mentioned transformations). Afterwards, we define the lift of a vector-valued function defined on the mesh domain onto the physical one, generalizing the definition given in~\cite{art-joyce-1} for a scalar function.

\paragraph{Curved mesh $\taur$ of order $r$.}
We denote $\tref$ the reference simplex of dimension~$d$. Let $\tauh$ be a polyhedral mesh of $\Omega$ made of simplices of dimension $d$, denoted~$T$ (triangles or tetrahedra). The mesh $\tauh$ is chosen as {quasi-uniform} and henceforth {shape-regular} (see  \cite[Definition 4.4.13]{quasi-unif} for more details). Denote $\fte$ the exact transformation, that maps the reference simplex $\tref$ into an exact mesh element consequently forming a mesh that exactly fits onto $\Omega$ as detailed in Appendix~\ref{mesh:appendix}. Then,~$\fte$ is interpolated as a polynomial of order $r \ge 1$ in the classical $\P r$-Lagrange basis on $\tref$. The interpolant is denoted by $\ftr$, which is a $\c1$-diffeomorphism and is in $\c {r+1}(\tref)$ (see \cite[chap. 4.3]{PHcia}). \medskip

Hence, the curved mesh of order~$r$ is denoted by~$\taur := \{ \tr:= \ftr(\tref); \, T \in \tauh \}$.  Additionally, $\omhh := \cup_{\tr \in  \taur}\tr$ denotes the mesh domain and~$\ghh:= \partial \omhh$ is its boundary. %We denote the outer unit normal vector over~$\ghr= \partial \omhr$ by $\nnh$.
\paragraph{Functional lift.}
A function defined on the curved mesh domain $\omhh$ can be lifted onto the physical domain $\Omega$, following the definitions first introduced in the 1970's (e.g., \cite{nedelec, scott, Lenoir1986, Bernardi1989}). In this paper, we employ the well-defined lift transformation $\Ghr$ introduced in \cite{art-joyce-1} to perform this operation. The transformation $\Ghr$ is defined piece-wise on each mesh element such that, 
\begin{equation}
\label{Ghr-ref}
    \Ghr : \omhr \to \Omega; \quad {\Ghr}_{|_{\ghr}} = b,
\end{equation}
where~$b$ is the orthogonal projection on the domain boundary $\Gamma$ defined in Proposition \ref{tub_neigh_orth_proj_prop}. We also mention that by construction, $\Ghr$ is globally continuous and piece-wise differentiable on each mesh element. We refer to Appendix \ref{appendix:lift} for the full expression of $\Ghr$. \medskip

\blue{Using this transformation, the volume lift operator for vector-valued functions is given in the following definition.}

\begin{definition}[Volume lift]
\label{def:liftvolume}
    To any vector-valued function $\uh \in \Lvect2(\omhr)$ is associated its lift, denoted~$\uhlifte \in \Lvect2(\Omega)$, given by, 
    $$
        \uhlifte \circ \Ghr = ({\uhlifte}_1 \circ \Ghr, \dots , {\uhlifte}_d \circ \Ghr) := \uh.
    $$
Similarly, to any vector-valued function $\u \in \Lvect2(\Omega)$, we can define its inverse lift, denoted~$\u^{-\ell} \in \Lvect2(\omhr)$, given by, 
    $$
       \u^{-\ell} := \u \circ \Ghr = (\u_1 \circ \Ghr, \dots , \u_d \circ \Ghr).
    $$
\end{definition}

\blue{
In a similar manner, a surface lift operator for vector-valued functions is introduced in the following definition.
\begin{definition}[Surface lift]
\label{def:liftsurface}
To any vector-valued function $\uh \in \Lvect2(\ghr)$, we associate its surface lift $\uh^L \in \Lvect2(\Gamma)$ defined by,
$$
\uh^L \circ b = ({\uh^L}_1 \circ b, \dots , {\uh^L}_d \circ b) := \uh,
$$
where $b: \ghr \rightarrow \Gamma$ is the orthogonal projection defined in Proposition~\ref{tub_neigh_orth_proj_prop}. Likewise, to any vector-valued function~$\u \in \Lvect2(\Gamma)$ is associated its inverse lift $\u^{-L} \in \Lvect2(\ghr)$ given by,
$$
\u^{-L} := \u \circ b = (\u_1 \circ b, \dots , \u_d \circ b).
$$
\end{definition}
}

\blue{
\begin{remarque}
This vector lift operator satisfies a trace property, which states that the volume and surface lifts coincide on~the boundary. In other words, for any vector field $\uh \in \Lvect2(\omhr)$, the following identity holds:
$$
\big( {{\uh}_i}_{|_{\ghr}} \big)^L = {\big( {\uhlifte}_i \big)}_{|_{\Gamma}}, \quad \forall \, i = 1, \dots, d.
$$
This follows from the fact that the restriction of the transformation $\Ghr$ to $\ghr$ coincides with the orthogonal projection $b$, i.e., ${\Ghr}_{|_{\ghr}} = b$. Consequently, the surface lift $\vh^L$ (resp. the inverse lift~$\v^{-L}$) will now simply be denoted by $\vhlifte$ (resp. $\v^{-\ell}$).
\end{remarque}
}

\subsection{The finite element approximation}
\label{sec:FEM}
\blue{Next, we introduce the finite element approximation of Problem~\eqref{fv_faible} using a $\mathbb{P}^k$ Lagrange finite element method (see~\cite{EG,PHcia}). This formulation is then lifted onto the physical domain~$\Omega$, leading to a lifted discrete problem whose solution approximates that of the continuous problem~\eqref{fv_faible}. For further details in the scalar setting, we refer to~\cite[\S5]{art-joyce-1}.}

\paragraph{The discrete formulation.} 
Recall that $\tref$ denotes the reference simplex of dimension~$d$. Given a curved mesh $\taur$,~$T$ denotes a curved mesh element. Let~$k \geq 1$, the $\P k$-Lagrangian finite element vector space is given by, 

$$
\Vhd := \{ \chi \in [\c0(\overline{\omhh})]^d; \ \chi_{|_T}= \hat{\chi} \circ (\ftr )^{-1} , \ \hat{\chi} \in [\P k(\hat{T})]^d, \ \forall \ T \in \taur \}.
$$

Given $\f \in \Lvect2(\Omega)$ and $\g \in \Lvect2(\Gamma)$ the right hand side functions of Problem~\eqref{sys:elast-Robin}, we define the discrete linear form $l_h$ on $\Vhd$, given as follows for $\vh \in \Vhd$,
\begin{equation}
\label{eq:lin-form-discrete}
     l_h(\vh) := \int_{\omhh} (\f^{-\ell}\Jh) \cdot \vh \,\dx + \int_{\ghh} (\g^{-\ell}\Jb) \cdot \vh \, \ds,
\end{equation}
where $\Jh$ (resp. $\Jb$) is the Jacobian of the lift transformation $\Ghr$ in \eqref{Ghr-ref} (resp. the orthogonal projection $b$ onto $\Gamma$ defined in Proposition \ref{tub_neigh_orth_proj_prop}) and $\f^{-\ell}$ (resp. $\g^{-\ell}$) is the inverse lift of $\f$ (resp.~$\g$).

\medskip
% 
% 
% With this definition, $l_h(\vh)=l(\vhlifte)$, for any $\vh \in \Vhd$, where $l$ is the right hand side in the weak formulation \eqref{fv_faible}. \medskip

The approximation problem is given by,
\begin{equation}
\label{fvh-EV}
\left\lbrace
      \begin{array}{l}
       \mbox{find } \uh \in \Vhd, \ \mbox{ such that,}  \\
         a_h(\uh,\vh ) = l_h(\vh ), \,  \forall \,\vh  \in \Vhd,
      \end{array}
    \right. 
\end{equation}
where $a_h$ is the following bilinear form, defined on $\Vhd \times \Vhd$, for~$\uh, \vh \in \Vhd$,
\begin{equation*}
    a_h(\uh,\vh ) :=  \int_{\omhh} \Aomega ( \euh ) :  \na \vh \, \dx + \int_{\ghh} \uh \cdot \vh \, \ds.
\end{equation*}

\begin{remark}
Since $a_h$ is bilinear symmetric positive definite on a finite dimensional space, then there exists a unique solution $\uh \in \Vhd$ to the discrete problem~\eqref{fvh-EV}.
\end{remark}

\paragraph{The lifted discrete formulation.} 
To define the lifted discrete formulation, we rely on the lifted finite element vector space defined given as follows, $\Vhlifted:= \{ \vhlifte ; \ \vh \in \Vhd \}.$ To begin with, we need to point out that the lifted finite element space $\Vhlifted$ is embedded in the Sobolev space~$\Hexpovect{1}(\Omega)$. Thus, the equations \eqref{pass_grad_volume-Elasticity}, \eqref{pass_fct_scalaire_volume-Elasticity} and \eqref{pass_fct_scalaire_surface-Elasticity}, where integrals on $\omhh$ (resp. $\ghh$) are expressed in terms of integrals on $\Omega$ (resp. $\Gamma$), can be applied for lifted finite element vector functions in the following. We refer to Section~\ref{sec:error-estimation-proof} for exhaustive details. \medskip

We define the lifted bilinear form $\ahell$, on $\Vhlifted \times \Vhlifted$, throughout,
$$
a_h(\uh,\vh ) = \ahell(\uhlifte,\vhlifte ), \quad \forall \, \uh,\vh  \in \Vhd.
$$  
By applying \eqref{pass_grad_volume-Elasticity}, and \eqref{pass_fct_scalaire_surface-Elasticity}, then the expression of $\ahell$ is given as follows for all $\uhlifte, \vhlifte \in~\Vhlifted$, 
\begin{multline*}
    \ahell(\uhlifte ,\vhlifte  ) =\muomega \int_{\Omega} (\na\uhlifte   \, \transGhr) : (\na\vhlifte   \, \transGhr) \, \frac{1}{\Jhlifte} \, \dx 
      + \muomega \int_{\Omega} (\na\uhlifte  \,  \transGhr){\transpose} : (\na\vhlifte  \,  \transGhr) \, \frac{1}{\Jhlifte} \, \dx \\
      + \lambdaomega \int_{\Omega} \Tr (\na\uhlifte  \,  \transGhr)  \, \Tr(\na\vhlifte  \,  \transGhr) \, \frac{1}{\Jhlifte} \, \dx + \int_{\Gamma}   \uhlifte   \cdot \vhlifte  \,   \frac{ 1}{\Jblifte} \, \ds,
\end{multline*}
where $\transGhr$ is a matrix arising from a change of variable using the vector lift operator. The full expression of $\transGhr$ is given in \eqref{notation:G-elasticity}. \medskip

Using \eqref{pass_fct_scalaire_volume-Elasticity} and \eqref{pass_fct_scalaire_surface-Elasticity}, we notice that the linear forms $l$ and $l_h$ given respectively in~\eqref{eq:lin-form-elast} and in \eqref{eq:lin-form-discrete} satisfy the following equation for all $\uhlifte, \vhlifte \in~\Vhlifted$,
\begin{multline*}
     l_h(\vh) = \int_{\omhh} (\f^{-\ell}\Jh) \cdot \vh \,\dx + \int_{\ghh} (\gh^{-\ell}\Jb) \cdot \vh \, \ds 
     = \int_{\Omega} \f\cdot \vhlifte \,\dx + \int_{\Gamma} \g\cdot \vhlifte \, \ds = l(\vhlifte).
\end{multline*}

Consequently, we define the lifted formulation of the discrete problem~\eqref{fvh-EV} by,
$$
\left\lbrace
      \begin{array}{l}
       \mbox{find } \uhlifte \in \Vhlifted, \ \mbox{ such that,}  \\
         \ahell(\uhlifte,\vhlifte)= l(\vhlifte), \,  \forall \,\vhlifte  \in \Vhlifted.
      \end{array}
    \right. 
$$
This problem is well-posed and admits a unique solution $\uhlifte \in \Vhlifted$, that is the lift of the unique solution $u_h$ of the discrete problem \eqref{fvh-EV}.
\begin{remarque}
    Keeping in mind that $\u$ is the solution of \eqref{fv_faible} and $\uhlifte$ is the lift of the solution of~\eqref{fvh-EV}, we need to point out that, for any $\vhlifte \in \Vhlifted \subset \Hexpovect{1}(\Omega)$, the following equation stands,
\begin{equation}
\label{rem:a=ahell-For-vhell}
  a(\u,\vhlifte) = l(\vhlifte)=  l_h(\vh)= a_h(\uh,\vh) = a^\ell_h(\uhlifte,\vhlifte).
\end{equation} 
\end{remarque}

\begin{remarque}
\label{rem:lame-coeff-cst}
Throughout this paper, the {\it Lamé} coefficients $\muomega > 0$ and~$\lambdaomega >0$ are supposed constant at each point $x$ of $\Omega$. We have to mention that they can be assumed to be variable if we suppose that they are bounded and superior to  a constant $\epsilon>0$. In this case, an additional technical difficulty arises: the {\it Lamé} coefficient associated to the discrete formulation of the problem need to be lifted from $\omhh$ onto $\Omega$. This is not a trivial difficulty to deal with that will not be held here.
\end{remarque}

%
%
% Sec error estimation 
%
%
\subsection{Main result: the error estimation theorem}
% \label{sec:error-estimation}

From this point forward, we consider that the mesh size~$h$ is sufficiently small and that $c$ refers to a positive constant independent of~$h$. Keeping in mind that the domain~$\Omega$, is assumed to be smooth (at least~$\c {k+1}$ regular with $k\ge 1$), we assume that the source terms in problem \eqref{sys:elast-Robin} are more regular:~$\f \in \Hexpovect{k-1}(\Omega)$ and $\g \in \Hexpovect{k-1}(\Gamma)$. Then, the exact solution $\u$ of Problem~\eqref{sys:elast-Robin} is in~$\Hexpovect{k+1}(\Omega)$ satisfying the following classical energy inequality,
\begin{equation}
\label{ineq:energy}
     \|\u\|_{\Hexpovect{k+1} (\Omega)} \le c ( \|\f\|_{\Hexpovect{k-1}(\Omega)} + \|\g\|_{\Hexpovect{k-1}(\Gamma)}).
\end{equation}
We refer to \cite[Th. 6.3-6]{ciarlet2022mathematical} and \cite[\S 2.7]{hansbo2020analysis} for more details.  \medskip

The goal of this paper is to prove the following a priori error estimates, stated as follows.
%
% THEOREM
%
%
%
\begin{theorem}
\label{th-error-bound}
Let $\u \in \Hexpovect{k+1}(\Omega)$ be the solution of the variational problem~\eqref{fv_faible} satisfying~\eqref{ineq:energy} and let $\uh \in \Vhd$ be the solution of the finite element formulation~\eqref{fvh-EV}. Then for a sufficiently small $h$, there exists a mesh independent constant $c > 0$ such that, 

\begin{equation}
\label{errh1_errl2}
    \|\u-\uhlifte\|_{\Hexpovect{1}(\Omega)}  \le c ( h^k + h^{r+1/2})
\quad \mbox{ and } \quad
    \|\u-\uhlifte \|_{\Lvect2\omgam } \le c ( h^{k+1} + h^{r+1}),
\end{equation}
where $\uhlifte \in \Vhlifted$ denotes the lift of $\uh$ onto $\Omega$, given in Definition \ref{def:liftvolume} and where the $\Lvect2\omgam$ norm is defined as follows $\|\v\|_{\Lvect2\omgam}^2 := \|\v\|_{\Lvect2(\Omega)}^2 + \|\restriction{\v}{\Gamma}\|_{\Lvect2(\Gamma)}^2,$ for any $\v \in \Lvect2(\Omega)$ such that $\restriction{\v}{\Gamma} \in \Lvect2(\Gamma)$.
\end{theorem}
%
% STEPS
%
\blue{The errors in \eqref{errh1_errl2} are controlled by two main components: the finite element error, quantified by the interpolation estimate in Proposition \ref{prop:interpolation-ineq} and the geometrical error in Proposition~\ref{prop:err-geom}, which arises from approximating the physical domain with a mesh of order~$r$. In Section \ref{sec:h1-err}, the $\Hexpovect{1}(\Omega)$ error estimate is established, where both the interpolation inequality and the geometric error play a key role. Lastly, in Section \ref{sec:L2-err}, the $\Lvect2$ error estimate is derived using the same key ingredients (the geometric error estimation and the interpolation inequality) along side an Aubin-Nitsche type argument.}

\section{Numerical simulations}
\label{sec:numerical-ex-elas}
\blue{In this section are presented numerical results aimed to illustrate the convergence estimates of Theorem \ref{th-error-bound}. We perform these simulations in the two dimensional and three dimensional cases. The discrete problem~\eqref{fvh-EV} is implemented and solved using the finite element library CUMIN~\cite{cumin}, an open-source code written in Fortran 90, designed for the efficient numerical approximation of partial differential equations.}

\medskip

In dimension $2$, the direct solver MUMPS\footnote{MUltifrontal Massively Parallel sparse direct Solver \url{https://mumps-solver.org/index.php}} is considered allowing fast computations. In dimension $3$, memory requirements imposed a lighter method: a conjugate gradient with Jacobi preconditioning has been used. The tolerances has been set to very low values of $(1{\rm E{-14}})$: this generally allowed to compute accurately the numerical errors up to error values of ${\rm 1 E{-11}}$, which was necessary in order to well capture the convergence asymptotic regimes. 
%In dimension $3$, memory requirements impose a lighter method: we will use the GMRES solver (Generalized Minimal Residual, see \cite[Chapter~6.5]{saad2003iterative}). It is an iterative solver that requires only the storage of the system matrix and a limited number of vectors, unlike factorization methods, which require a large amount of memory and fail in $3$D. It is also important to note that our system is not symmetric, which promotes the choice of GMRES. As will be detailed in the upcoming section, this solver is considered in the three dimensional case with a left pre-conditioner and the tolerance for the algorithm has been set to very low values $(1{\rm E{-14}})$: this generally allowed to compute accurately the numerical errors up to error values of ${\rm 1 E{-11}}$, which was necessary in order to well capture the convergence asymptotic regimes.

\medskip

Curved meshes of the domain $\Omega$ of geometrical order~$1\le r \le 3$ have been generated using the software Gmsh%\footnote{Gmsh: a three-dimensional finite element mesh generator, \url{https://gmsh.info/}}
\footnote{\url{https://gmsh.info/}}. All integral computations (either on the physical domain $\Omega$ or on the computational domain~$\omhr$) are performed on the reference simplex using changes of coordinates. These changes of coordinates are made on each element of the considered mesh. This allows to compute numerical errors such as $\|\uhlifte - \u\|_{\Lvect2(\Omega)}$ between the lift $\uhlifte$ of a finite element vector-valued function~$\uh$ defined on $\omhr$ and a vector-valued function~$\u$ defined on the smooth domain $\Omega$. On the reference simplex, high-order quadrature methods are employed, ensuring that the integration error is of smaller magnitude than the approximation errors evaluated in this section. Indeed, it has been consistently verified that the integration errors have a negligible impact on the subsequent numerical results.  \medskip
 
\blue{In addition, all numerical results presented in this section are fully reproducible using the dedicated source codes available in the CUMIN GitLab repository\footnote{\url{https://plmlab.math.cnrs.fr/cpierre1/cumin}}.}

\subsection{The two dimensional case: on the unit disk}

The linear elasticity problem with Robin's boundary condition in \eqref{sys:elast-Robin} is considered on the unit disk ${\rm D(O,1)} \subset \R^2$, with the following Lamé constants, $\lambdaomega  = \muomega =1$. In this example, we consider the following source terms, 
$$
\mathbf{f}(x, y) = \begin{bmatrix} -6 y^2 - x \exp(y) \\ -18 xy - 2 \exp(y) \end{bmatrix} \quad \mbox{and} \quad  \mathbf{g}(x, y) = \begin{bmatrix} 3x^2y^2 + y^4 + xy \exp(y) + 4x \exp(y) \\ 11xy^3 + x^2 \exp(y) + y \exp(y) \end{bmatrix}.
$$
Hence, the analytical solution of \eqref{sys:elast-Robin} is given by, $\u(x,y) =  \begin{bmatrix} x\exp(y) \\ y^3x \end{bmatrix}$, for all $(x,y)$ in ${\rm D(0,1)}$. 

\medskip

\begin{figure}[!h]
\centering
\begin{tabular}{c c} 
    \includegraphics[width=0.27\textwidth]{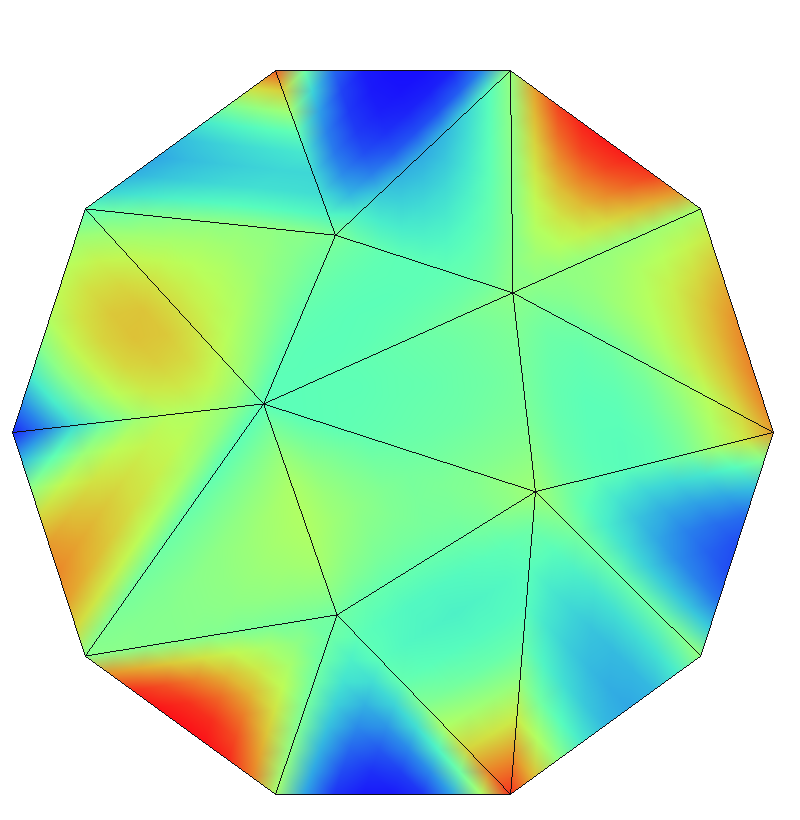}
    &
    \includegraphics[width=0.28\textwidth]{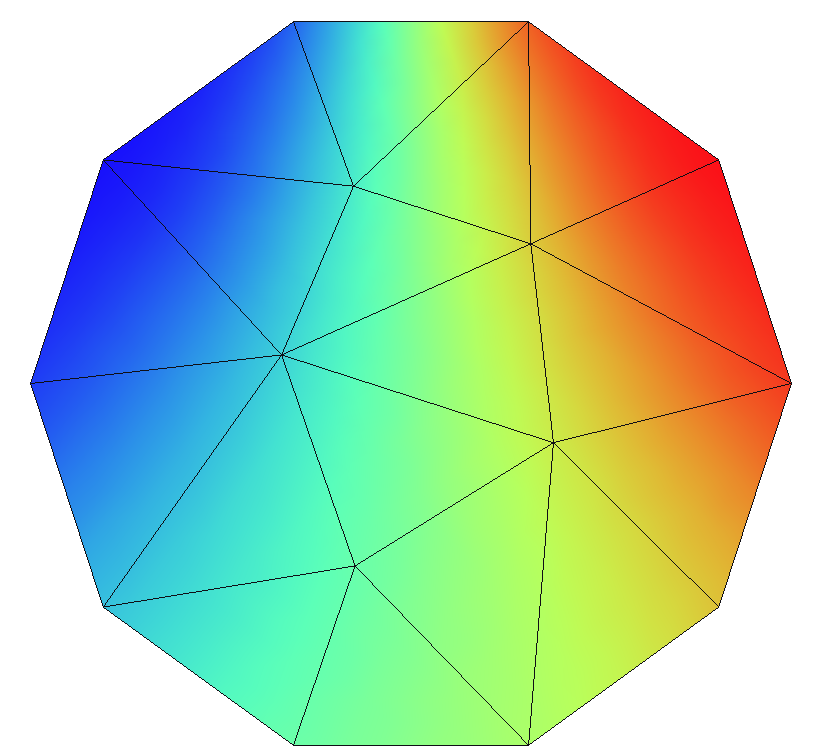} \\
    \includegraphics[width=0.27\textwidth]{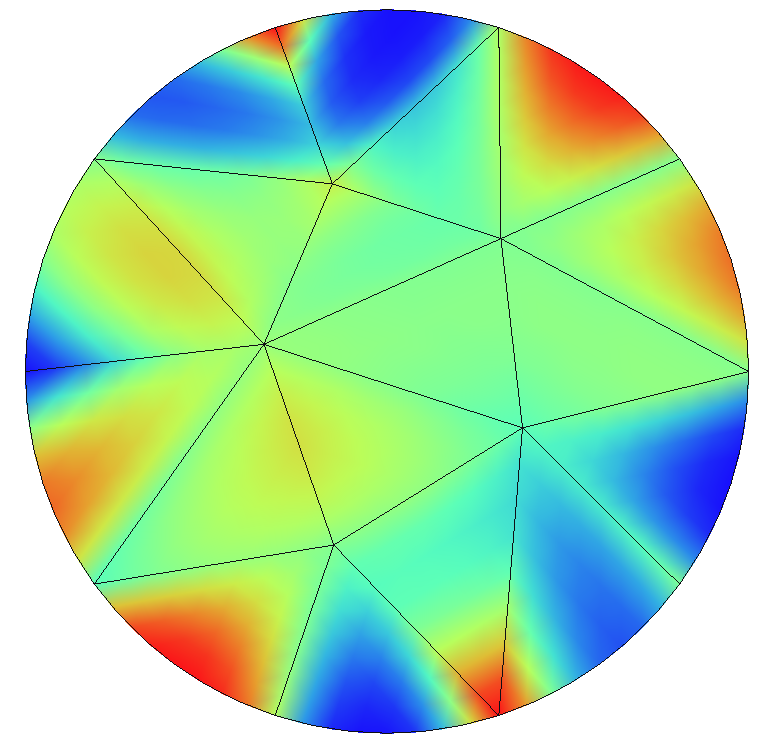}
    &
    \includegraphics[width=0.27\textwidth]{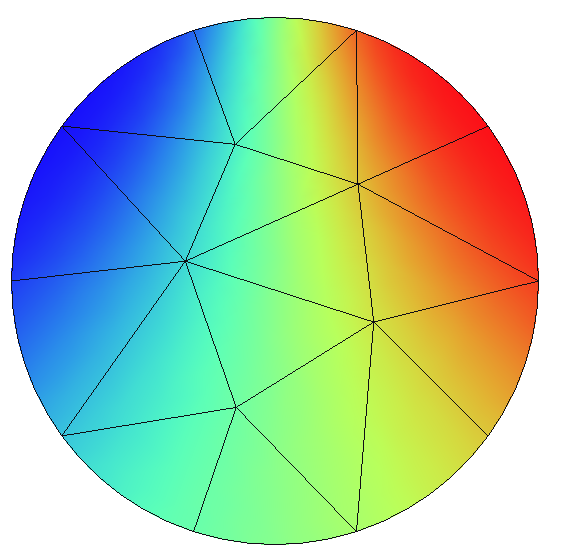}  
\end{tabular}
 \caption{\centering \blue{Visualization of the first (left) and second (right) components of the numerical solution to the linear elasticity problem, computed using a $\P 2$ finite element method on both affine and quadratic meshes.}}
\label{fig:disk-Robin}
\end{figure}
The numerical solutions $\uh$ are computed using $\mathbb{P}^k$ finite elements for $k=1,\dots, 4$, on a series of successively refined meshes of orders $r=1,\dots, 3$. Each mesh has $10 \times 2^{n-1}$ boundary edges, with~$n=1,\dots, 7$. On the finest mesh, we have $10 \times 2^6$ boundary edges and approximately with a $\P 4$ finite element method $75 \, 500$ triangles. The corresponding $\P 4$ finite element space has an approximate dimension of $2 \times 605 \, 600$. In Figure~\ref{fig:disk-Robin}, the components of the numerical solution~$\uh = (u_1, u_2)$ computed using a $\P2$ method are shown on both affine and quadratic coarse meshes. 

\medskip

For each mesh order $r$ and each finite element degree $k$, the following numerical errors are computed:
$$
   \eL2 := \| \u-\uhlifte \|_{\Lvect2(\Omega)}, \quad 
   \eH1 := \| \nabla \u- \nabla \uhlifte \|_{\Lvect2(\Omega) },\quad {\rm and} \quad
  \eLGamma := \|  \u-\uhlifte \|_{\Lvect2(\Gamma)}.
$$
The convergence orders of these errors, interpreted in terms of the mesh size $h$, are reported in Table~\ref{tab:Robin-2d-omega}~and in Figure \ref{fig:Robin-2d-omega} for the volume errors and in Table~\ref{tab:Robin-2d-L2-Gamma}~and in Figure \ref{fig:Robin-2d-L2-Gamma} for the surface error. The convergence order are evaluated from the error ratio between two successive meshes. 

\medskip

Before discussing the results obtained, we recall that the \textit{a priori} error estimates given in Theorem \ref{th-error-bound} can be written as follows,
\begin{equation}
  \label{Robin-error-estimate}
    \eL2 = O(h^{k+1} + h^{r+1}), \quad \eH1 = O(h^k+h^{r+1/2}) \quad \mbox{and} \quad \eLGamma = O(h^{k+1} + h^{r+1}).
\end{equation}

\begin{table}[!ht]
    \centering
    \begin{tabular}{|l||l|l|l|l||l|l|l|l|}%{|l|M|C|R|J|}
\cline{2-9}
\multicolumn{1}{c||}{}    &  \multicolumn{4}{|c||}{$\eL2$} & \multicolumn{4}{|c|}{$\eH1$}  \\[0.08cm]
\cline{2-9}
\multicolumn{1}{c||}{}    &  $\P1$ &  $\P2$ &  $\P3$ &  $\P4$ & $\P1$ &  $\P2$ &  $\P3$ &  $\P4$  \tabularnewline
\hline
 Affine mesh (r=1)   & 
 {2.00} &  {2.00} &  {2.00} &  {2.00} & 
1.01 &  {1.51} &  {1.50} &  {1.50} \tabularnewline
\hline
Quadratic mesh (r=2) & 
2.01 & 3.04 &  \textcolor{blue}{3.98} &  \textcolor{blue}{4.01} & 1.00 & 2.04 & \textcolor{blue}{2.97} &  \textcolor{blue}{3.50} \tabularnewline
\hline
Cubic mesh (r=3) & 
2.04 &  \textcolor{red}{2.48} &  \textcolor{red}{3.48} & 4.00 & 
1.01 &  \textcolor{red}{1.49} &  \textcolor{red}{2.47} & 3.49 \tabularnewline
\hline
\end{tabular}

\caption{
 \blue{Convergence orders of the volume errors (numbers highlighted in red indicate a loss in convergence rate, while those in blue denote super-convergence of the errors).}
  \label{tab:Robin-2d-omega}
}
\end{table}

\begin{figure}[h]
    \begin{center}
    \begin{tabular}{ c c } 
 \includegraphics[width=5.5cm, height=4.5cm]{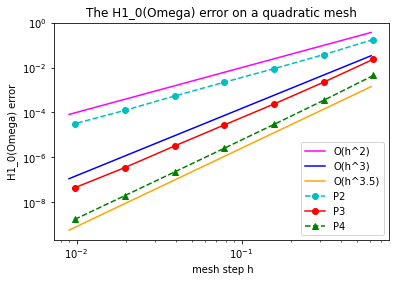} & \includegraphics[width=5.5cm, height=4.5cm]{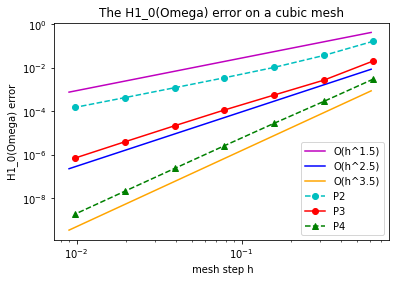}  \\ \includegraphics[width=5.5cm, height=4.5cm]{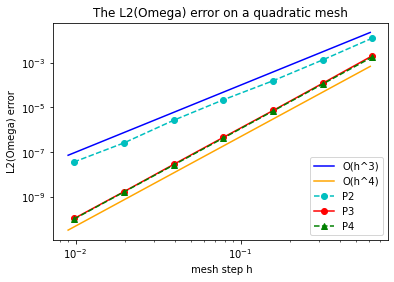} &
 \includegraphics[width=5.5cm, height=4.5cm]{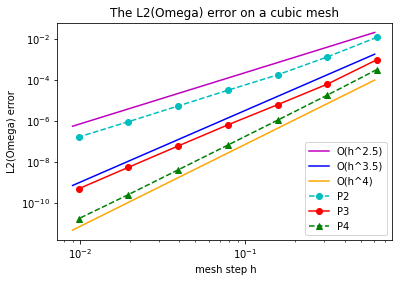}
\end{tabular}
\end{center}
     \caption{\blue{Plots of the error in volume norms with respect to the mesh size $h$, corresponding to the convergence orders reported in Table~\ref{tab:Robin-2d-omega}: $\eH1$ (top) and $\eL2$ (bottom) for quadratic meshes (left) and cubic meshes (right).}}
     \label{fig:Robin-2d-omega}
 \end{figure}
 
The convergence orders presented in Table \ref{tab:Robin-2d-omega} and in Figure \ref{fig:Robin-2d-omega}, relatively to the $\Lvect2$ and $\Hexpovect{1}_0$ norms on $\Omega$, deserve exhaustive comments. In the affine case $r=1$, the figures are in perfect agreement with the estimates~\eqref{Robin-error-estimate}: $\eL2$ is in $O(h^{k+1} + h^2)$ and  $\eH1$ is in  
$O(h^{k} + h^{1.5})$.

For quadratic meshes, a super convergence highlighted in blue is observed in the geometric error, the case $r=2$ behaves as if  $r=3$: $\eL2$ is in $O(h^{k+1} + h^4)$ and $\eH1$ is in $O(h^{k} + h^{3.5})$. {This is quite visible on the bottom left of Figure \ref{fig:Robin-2d-omega} for $\eL2$: while using respectively a $\P 3$ and~$\P 4$ method, the $\Lvect2$ error graphs in both cases follow the same line representing $O(h^4)$. In the case of the $\Lvect2$ gradient norm of the error, this super convergence is depicted with a $\P3$ (resp. $\P4$) method: the convergence order is equal to 3 (resp. 3.5) surpassing the expected value of $2.5$.}
This super convergence, though not understood, has been documented and further investigated in \cite{art-joyce-1,art-joyce-2,D4,Jaca}. Additional numerical investigations in \cite[Chapter 4]{these-J.GH} demonstrated that the geometric error associated with quadratic meshes for integral computations scales as $O(h^4)$ across various non-convex, asymmetric domains in both $2$D and $3$D. This behavior appears to be neither specific to the current problem, nor to the disk geometry considered here, nor dependent on the domain dimension.

For the cubic case, following \eqref{Robin-error-estimate}, $\eL2$ is expected to be in $O(h^{k+1} + h^4)$ and $\eH1$ in~$O(h^{k} + h^{3.5})$. This is accurately observed for a $\P1$ (resp.~$\mathbb{P}^4$) method: the $\Lvect2$ error is equal to $2.04$ (resp.~$4.00$) and the $\Lvect2$ gradient error is equal to~$1.01$ (resp.~$3.49$). However, a default of order~$-1/2$ is observed on the convergence orders in the $\mathbb{P}^2$ and $\mathbb{P}^3$ case. \blue{Remark~\ref{rk:cubic} discusses this behavior, emphasizing that it is neither caused by the lift operator nor by the specific problem considered. Instead, it is related to the finite element interpolation error. This error in the $\Hexpovect{1}(\Omega)$ norm behaves like $O(h^{k-1/2})$ instead of $O(h^k)$ for $k \ge 2$. Similarly, the associated $\Lvect2(\Omega)$ norm behaves like $O(h^{k+1/2})$ instead of $O(h^{k+1})$ for $k \ge 2$. }
\medskip 

\begin{table}[!ht]
    \centering
    \begin{tabular}{|l||l|l|l|l|}%{|l|M|C|R|J|}
\cline{2-5}
\multicolumn{1}{c||}{}    &  \multicolumn{4}{|c|}{$\eLGamma$}  \\[0.08cm]
\cline{2-5}
\multicolumn{1}{c||}{}    &  $\P1$ &  $\P2$ &  $\P3$ &  $\P4$ \tabularnewline
\hline
 Affine mesh (r=1)   & 
 {1.98} &  {1.99} &  {1.99} &  {1.99} \tabularnewline
\hline
Quadratic mesh (r=2) & 
2.08 & 3.01 &  \textcolor{blue}{3.99} &  \textcolor{blue}{4.01}  \tabularnewline
\hline
Cubic mesh (r=3) & 
2.09 &  \textcolor{red}{2.01} &  \textcolor{red}{2.95} & 4.00 \tabularnewline
\hline
\end{tabular}
\caption{
  \blue{Convergence orders of the surface error (numbers highlighted in red indicate a loss in convergence rate, while those in blue denote super-convergence of the errors).}
  \label{tab:Robin-2d-L2-Gamma}
}
\end{table}

\begin{figure}[h]
\begin{center}
\begin{tabular}{ c c } 
\includegraphics[width=5.5cm, height=4.5cm]{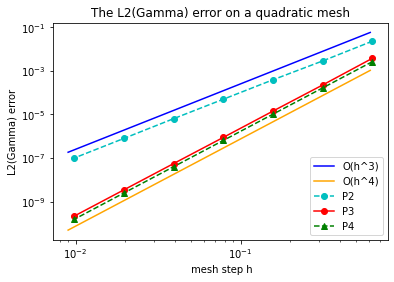} &
\includegraphics[width=5.5cm, height=4.5cm]{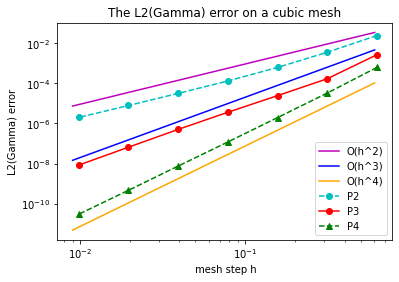} 
\end{tabular}
\end{center}
    \caption{
\blue{Plots of the error $\eLGamma$ with respect to the mesh size $h$, corresponding to the convergence orders reported in Table~\ref{tab:Robin-2d-omega}, for quadratic meshes (left) and cubic meshes (right).
}}
\label{fig:Robin-2d-L2-Gamma}
\end{figure}
 
Let us now discuss Table \ref{tab:Robin-2d-L2-Gamma} and Figure \ref{fig:Robin-2d-L2-Gamma}, where the surface error $\eLGamma$ and its convergence rates are observed. The $\Lvect2$ surface error behaves as expected, following \eqref{Robin-error-estimate}. Indeed, in the affine case,~$\eLGamma$ is in $O(h^{k+1} + h^{2})$, for any $k \ge 1$. On the quadratic meshes, the super-convergence previously mentioned is clearly visible for the surface error. For a $\P k$ method with $k \ge 3$, $\eLGamma$ is in~$O(h^{4})$ instead of $O(h^{3})$. On the cubic meshes, we notice a default of convergence of magnitude~$-1$, for a $\P2$ (resp. $\P3$) method. This finite element error loss is more than the one observed for volume errors. This new unexpected behavior is not observed in the scalar cases in \cite{art-joyce-1,art-joyce-2} and seems then coming from this vectorial case. 

\blue{
\begin{remark}
\label{rk:cubic}
    Based on the results presented in Table~\ref{tab:Robin-2d-omega} and Table~\ref{tab:Robin-2d-L2-Gamma}, the error estimates exhibit optimal convergence rates on linear and quadratic meshes. As depicted on cubic meshes ($r = 3$), these errors present the following behaviors when using $\mathbb{P}^2$ and $\mathbb{P}^3$ finite element methods:
    $$
    \|\u-\uhlifte\|_{\Hexpovect{1}(\Omega)}  \le c ( h^{k-1/2} + h^{3.5}), 
    $$
    $$
    \|\u-\uhlifte \|_{\Lvect2(\Omega) } \le c ( h^{k+1/2} + h^{4}) \quad \mbox{ and } \quad
    \|\u-\uhlifte \|_{\Lvect2(\Gamma) } \le c ( h^{k} + h^{4
    }),
    $$
    where $k =2, 3$. The above errors rely on two main components: the geometric error estimate proved in Proposition~\ref{prop:err-geom}, and the lifted interpolation error inequality stated in Proposition~\ref{prop:interpolation-ineq}. On cubic meshes, when using a~$\mathbb{P}^2$ (respectively $\mathbb{P}^3$) method, a default in the convergence rate appears to occur only in the interpolation error. To investigate this unexpected behavior, a series of tests were conducted in both 2D and 3D, as described in \cite[Section 4.5]{these-J.GH}, focusing on estimating the difference between a smooth function and its finite element interpolant. This analysis revealed that the issue is not problem-dependent, as it arises in various problems. Furthermore, these tests also proved that the loss in convergence is independent of both the lift operator used and the geometric error. While conducting some experiments, we noticed that this interpolation error is highly sensitive to the position of the central node in cubic elements without being able so far to overcome this issue.
\end{remark}
}

\subsection{The three dimensional case: on the unit ball}
Next, we consider the linear elasticity problem \eqref{sys:elast-Robin} with Robin boundary conditions on the unit ball ${\rm B(O,1)} \subset \R^3$, using the Lamé constants $\lambda_\Omega = \mu_\Omega = 1$. In this example, the source term on the unit ball is given as follows,
\begin{center}
    $\f(x,y,z) = 
\begin{bmatrix}
- 6 y^2 z^2 - x \exp(y) \\
-2xy^3 - 8z^3 - 18 xyz^2 - 2 \exp(y) \\
-12x y^2 z - 36 y z^2
\end{bmatrix},
$
\end{center}
and the source term on the unit sphere is defined as follows,
\begin{center}
    $\g(x,y,z) = 
\begin{bmatrix}
3x^2y^2z^2 + y^4z^2 + xy \exp(y) + 4x \exp(y) + 4 x y z^3 \\
13xy^3z^2 + x^2 \exp(y) + y \exp(y) + 4y^2z^3 + z^5 \\
2xy^4z + 3xy^2z^3 + 14 y z^4 + z \exp(y)
\end{bmatrix}.$
\end{center} 
The analytical solution of this problem thus is given by, $\u(x,y,z) = 
\begin{bmatrix}
x \exp(y) \\
xy^3 z^2 \\
yz^4
\end{bmatrix}$. %, for all $(x,y,z)$ in ${\rm B} (O,1).$

\medskip
\begin{figure}[!h]
\centering
\begin{tabular}{c c c} 
    \includegraphics[width=0.285\textwidth]{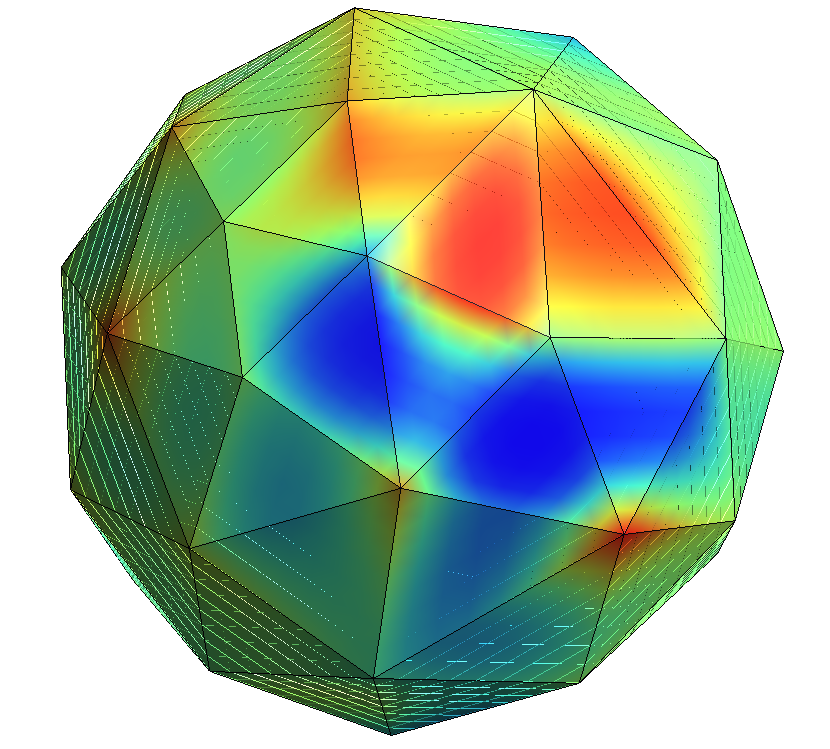}
    &
    \includegraphics[width=0.265\textwidth]{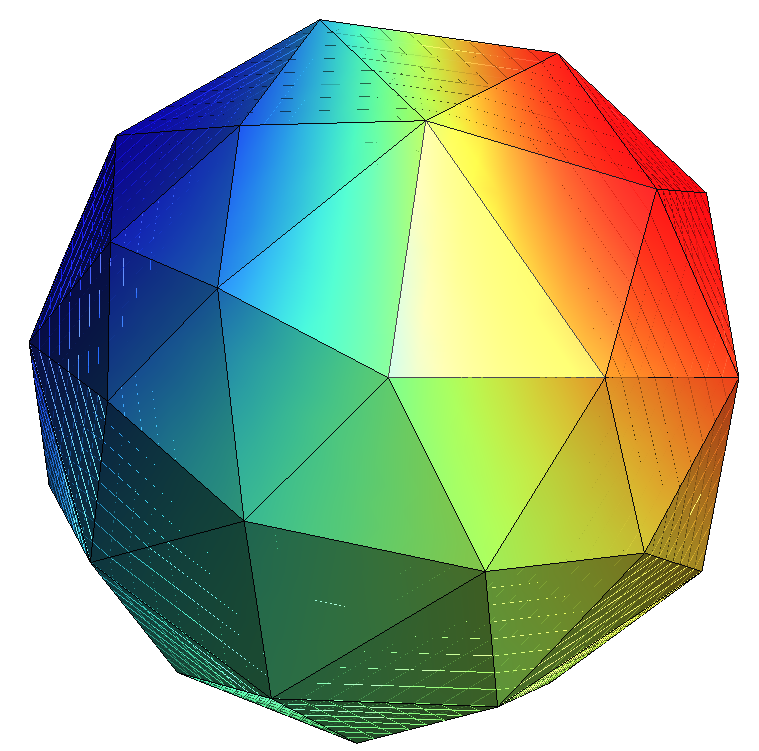} &
    \includegraphics[width=0.265\textwidth]{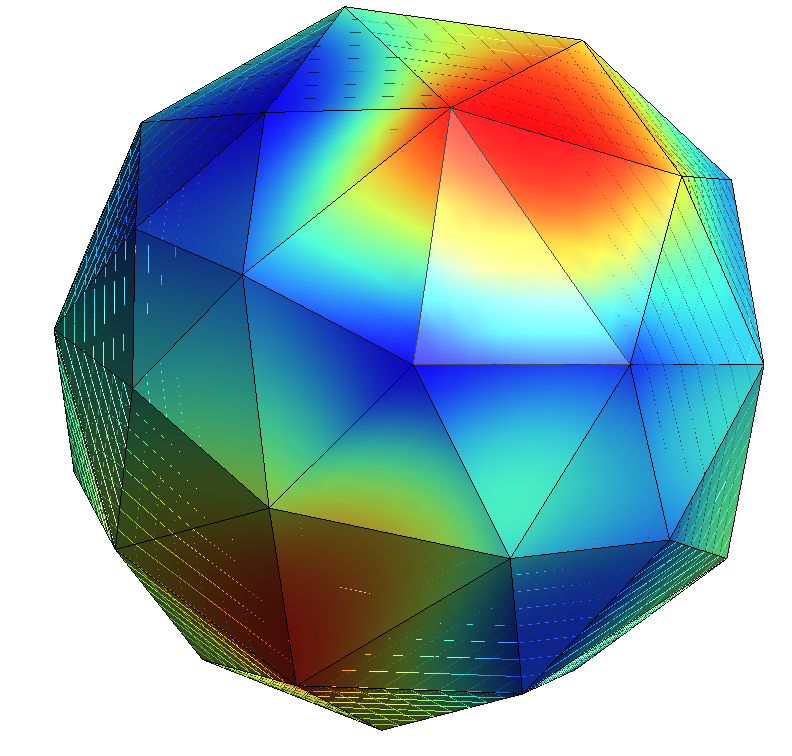}
    \\
    \includegraphics[width=0.27\textwidth]{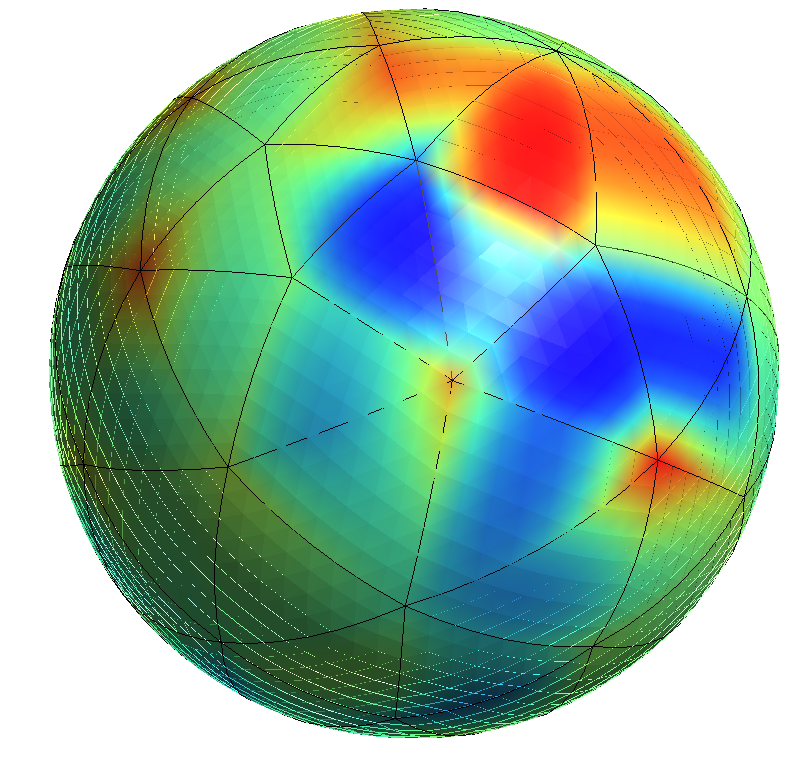} &
    \includegraphics[width=0.27\textwidth]{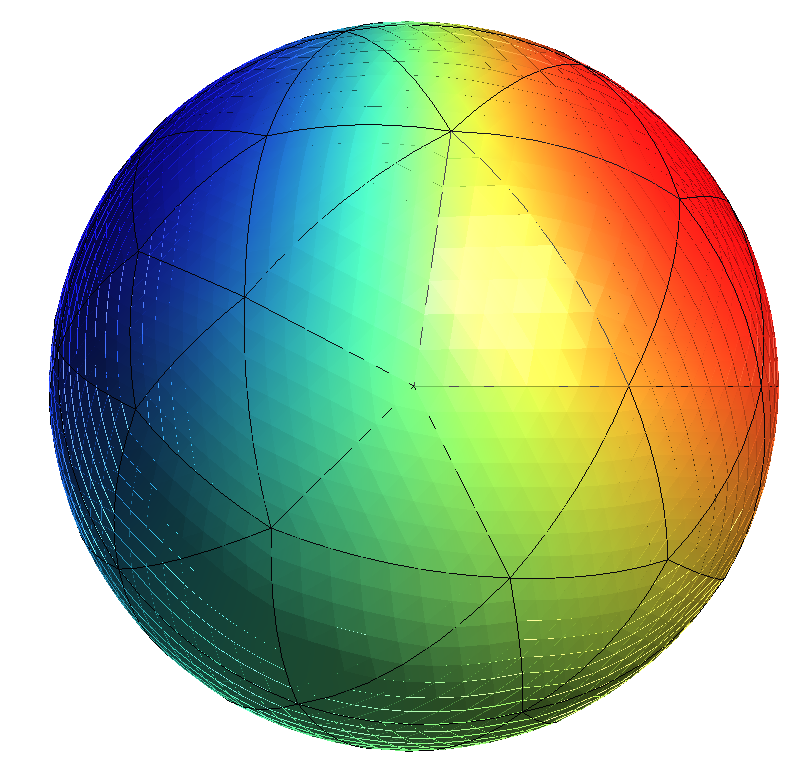} &
    \includegraphics[width=0.27\textwidth]{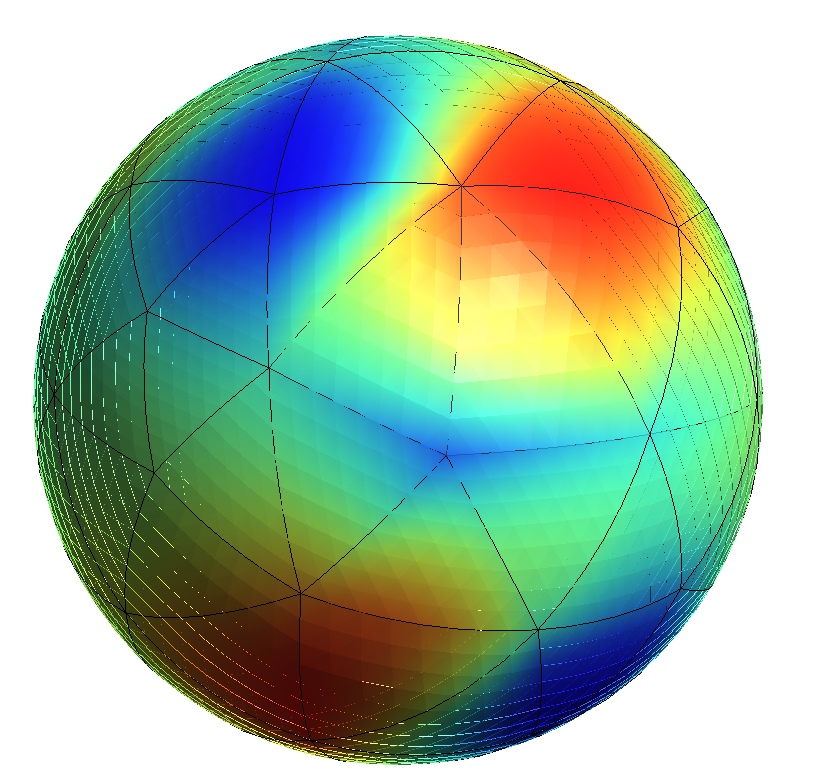} 
\end{tabular}
 \caption{\centering \blue{Visualization of the first (left), second (center), and third (right) components of the numerical solution to the linear elasticity problem, computed using a $\P 2$ finite element method on affine and quadratic meshes.}}
% \label{fig:ball-Robin}
\end{figure}

For each mesh order $r$ and finite element degree $k$, we compute the error on a series of five successively refined meshes. Each mesh counts~$10\times 2^{n-1}$ edges on the equator circle, for~$n=1, \dots ,5$. The most refined mesh has approximately $3,1 \times 10^5$ tetrahedra and the associated $\P 3$ finite element method counts  $1,5 \times 10^6$ degrees of freedom. In the $3$D case, the computations are much more demanding. The inversion of the linear system is done using the conjugate gradient method with a Jacobi pre-conditioner. %The inversion of the linear system is done using GMRES with a left pre-conditioner: indeed, we use an incomplete LU factorization of order~$0$ (see \cite{saad2003iterative}) with a tolerance at $1{\rm E}-14$ preceded by a renumbering of points to produce a matrix more concentrated along the diagonal (which improves the quality of the preconditioner). The renumbering algorithm applied here is the Reverse Cuthill-McKee algorithm (see \cite{CuthillMcKee1969}). 
To handle these time consuming computations, we resorted to the UPPA research computer cluster PYRENE\footnote{PYRENE Mesocentre de Calcul Intensif Aquitain, \url{https://git.univ-pau.fr/num-as/pyrene-cluster}.}, while using shared memory parallelism on a single CPU with~$32$~cores and~$2\, 000$~Mb of memory. \medskip

As in the $2$D case, we evaluate the following errors for each mesh order $r$ and each finite element degree $k$, 
$$
   \eL2 := \| \u-\uhlifte \|_{\Lvect2(\Omega)}, \quad 
   \eH1 := \| \nabla \u- \nabla \uhlifte \|_{\Lvect2(\Omega) },\quad {\rm and} \quad
  \eLGamma := \|  \u-\uhlifte \|_{\Lvect2(\Gamma)}.
$$
In Table~\ref{tab:Robin-Omega-3d}~and in Table~\ref{tab:Robin-Gamma-3d}, we input the obtained error convergence orders evaluated from the error ratio between two successive meshes. As a general comment: similar results as in the $2$D case are observed and the quadratic meshes also exhibit a super-convergence (highlighted in blue).% behaving as if $r=3$ instead of the expected~$r=2$.
\medskip

\begin{table}[!ht]
    \centering
    \begin{tabular}{|l||l|l|l||l|l|l|}%{|l|M|C|R|J|}
\cline{2-7}
\multicolumn{1}{c||}{}    &  \multicolumn{3}{|c||}{$\eL2$} & \multicolumn{3}{|c|}{$\eH1$}  \\[0.08cm]
\cline{2-7}
\multicolumn{1}{c||}{}    &  $\P1$ &  $\P2$ &  $\P3$ & $\P1$ &  $\P2$ &  $\P3$ \tabularnewline
\hline
 Affine mesh (r=1)   & 
 {2.00} &  {2.00} &  {2.00} &  
1.01 &  {1.51} &  {1.50}  \tabularnewline
\hline
Quadratic mesh (r=2) & 
2.01 & 3.04 &  \textcolor{blue}{3.98} & 1.00 & 2.04 & \textcolor{blue}{2.97}  \tabularnewline
\hline
Cubic mesh (r=3) & 
2.04 &  \textcolor{red}{2.48} &  \textcolor{red}{3.48} & 
1.01 &  \textcolor{red}{1.49} &  \textcolor{red}{2.47}  \tabularnewline
\hline
\end{tabular}

\caption{
  \blue{Convergence orders of the volume errors (numbers highlighted in red indicate a loss in convergence rate, while those in blue denote super-convergence of the errors).} 
  \label{tab:Robin-Omega-3d}
}
\end{table}

\blue{Table~\ref{tab:Robin-Omega-3d} presents results similar to those obtained in the disk case: both the $\eH1$ and $\eL2$ errors follow the estimate given in \eqref{Robin-error-estimate}. In the cubic case, the same loss of $-1/2$ in the convergence rate, as discussed in Remark~\ref{rk:cubic}, is observed.} Indeed, the $\Hexpovect{1}(\Omega)$ error is in $O(h^{1.49})$ (resp. $O(h^{2.47})$) for a $\P2$ (resp. $\P 3$) method and the $\Lvect2(\Omega)$ error is in $O(h^{2.48})$ (resp.~$O(h^{3.48})$) for a $\P2$ (resp. $\P 3$).
\medskip

\begin{table}[!ht]
    \centering
    \begin{tabular}{|l||l|l|l|}%{|l|M|C|R|J|}
\cline{2-4}
\multicolumn{1}{c||}{}    &  \multicolumn{3}{|c|}{$\eLGamma$}  \\[0.08cm]
\cline{2-4}
\multicolumn{1}{c||}{}    &  $\P1$ &  $\P2$ &  $\P3$ \tabularnewline
\hline
 Affine mesh (r=1)   & 
 {1.98} &  {1.99} &  {1.99} \tabularnewline
\hline
Quadratic mesh (r=2) & 
2.08 & 3.01 &  \textcolor{blue}{3.99}   \tabularnewline
\hline
Cubic mesh (r=3) & 
2.09 &  \textcolor{red}{2.01} &  \textcolor{red}{2.95}  \tabularnewline
\hline
\end{tabular}
\caption{
  \blue{Convergence orders in the surface norm (numbers highlighted in red indicate a loss in convergence rate, while those in blue denote super-convergence of the errors).}
  \label{tab:Robin-Gamma-3d}
}
\end{table}

\blue{As observed in the disk case, the $\Lvect2$ surface error in Table~\ref{tab:Robin-Gamma-3d} follows the expected behavior described in Equations~\eqref{Robin-error-estimate}. It is worth mentioning that, as in the 2D case on cubic meshes following Remark~\ref{rk:cubic}, a loss of one order in the convergence rate is observed with the $\P2$ (respectively $\P3$) method: we obtain a rate of 2 (resp. 3) instead of the expected 3 (resp. 4).}

\section{Proof of Theorem \ref{th-error-bound}}
\label{sec:error-estimation-proof}

\blue{This section is devoted to the derivation of the error estimates. We begin by introducing several preliminary tools essential for this analysis. First, we provide a detailed expression of the bilinear form $a$ appearing in the weak formulation \eqref{fv_faible}. In the context of linear elasticity, the bilinear form involves tensorial quantities and derivatives of vector-valued functions, which are particularly intricate to handle. We then employ the lift operator to express integrals over the discrete domain in terms of integrals over the physical domain. These expressions allow us to define the lifted bilinear form $\ahell$. Additionally, we recall several key estimates and properties associated with the lift transformation $\Ghr$, which are crucial for controlling the geometric approximation errors.}

\paragraph{The expression of the bilinear form.}
The following proposition gives a detailed expression of the bilinear form $a$. 
\begin{proposition}
\label{prop:a-express}
    The bilinear form in the weak formulation \eqref{fv_faible}, $a:\H\times\H\to\R$, can be expressed as follows,
\begin{equation*}
a(\u,\v)  = \muomega \( \int_{\Omega} \na\u : \na\v \, \dx + \int_{\Omega} (\na\u)^{\transpose} : \na\v \, \dx \) + \lambdaomega \int_{\Omega} \div (\u)  \,\div(\v) \, \dx  + \int_{\Gamma} \u\cdot \v \, \ds,
\end{equation*}
for any $\u, \v \in \H$.
\end{proposition}

\begin{proof}
  The proof is based on the following equations, which state that, for any $\u, \v \in \H$,
\begin{align}
  \Aomega (\eu) : \ev &= \Aomega (\eu) : \na \v, \label{rk:Aomega-sym} \\
   &= \muomega\big( \na \u : \na \v + (\na \u)^{\transpose} : \na \v \big) + \lambdaomega \div(\u) \, \div(\v). \label{rk:Aomega_eu:na_v}
\end{align}

We start by proving Equation \eqref{rk:Aomega-sym}. By definition of the strain tensor in \eqref{eq:e(u)}, we have,
    \begin{align*}
    \Aomega (\eu): \ev  &= \frac{1}{2} \bigg( \Aomega (\eu): \na \v +  \Aomega (\eu): (\na \v){\transpose} \bigg)  \\
    & = \frac{1}{2} \bigg( \Aomega (\eu): \na \v +  ( \Aomega (\eu) ){\transpose}: \na \v \bigg).    
    \end{align*}
    Since $( \Aomega (\eu) ){\transpose} = \Aomega (\eu)$ for any $\u \in \H$, we obtain \eqref{rk:Aomega-sym}. 
    \medskip

To prove \eqref{rk:Aomega_eu:na_v}, we use the definition of the Hooke tensor in \eqref{eq:Hooke-tensor} and  of the strain tensor in~\eqref{eq:e(u)} as follows,
\begin{align*}
    \Aomega (\eu): \na \v   &= 2\muomega \, \eu : \na \v \, + \lambdaomega \, \Tr(\na \u)\, \I : \na \v \\
    &= \muomega \big( \na \u: \na \v + (\na \u)^{\transpose}: \na \v \big) + \lambdaomega \, \Tr(\na \u)\, \Tr( \na \v) \\
    & = \muomega \big( \na \u: \na \v + (\na \u)^{\transpose}: \na \v \big) + \lambdaomega \,  \div ( \u) \, \div (\v).
    \end{align*}
\end{proof}   

% \subsubsection{Surface and volume integrals expression}
% Firstly, integrals over $\ghr$ (resp.~$\omhr$) are expressed with respect to ones over $\Gamma$ (resp. $\Omega$), in the context of the elasticity problem~\eqref{sys:elast-Robin} using the lift operator given in Definition \ref{def:liftvolume}. 

\paragraph{Surface integrals.}
Let~$\uh, \vh \in \Lvect2 (\ghr) $ with $\uhlifte, \vhlifte \in \Lvect2 (\Gamma)$ as their respected lifts. Then, the following integral over $\ghr$ can be expressed with respect to an integral over $\Gamma$ as follows,
\begin{equation}
    \label{pass_fct_scalaire_surface-Elasticity}
    \int_{\ghr}  \uh \cdot \vh  \, \ds = \int_{\Gamma}   \uhlifte \cdot \vhlifte   \frac{ 1}{\Jblifte} \, \ds,
\end{equation}
where \blue{$\Jb$ denotes the tangential Jacobian of $b$ the orthogonal projection on the boundary defined in \cite[Equation (2.10)]{D1}}, and $\Jblifte$ is its lift.% given by $\Jblifte \circ b = \Jb$.
\medskip
%
%

%
% Passage entre integrale 
%
%
\paragraph{Volume integrals.} In a similar manner, consider $\uh, \vh \in \Hexpovect{1} (\omhr)$ and let $\uhlifte, \vhlifte \in \Hexpovect{1} (\Omega)$ be their respected lifts, we have,
\begin{equation}
\label{pass_fct_scalaire_volume-Elasticity}
    \int_{\omhr}\uh \cdot \vh \, \dx = \int_\Omega \uhlifte \cdot \vhlifte \frac{1}{\Jhlifte} \dx,
\end{equation} 
where $\Jh$ denotes the Jacobian of $\Ghr$ and $\Jhlifte$ is its lift. % given by $\Jhlifte \circ \Ghr = \Jh$. 
\medskip

Moreover, we define an expression for the lift of the gradient of a discrete function $\vh \in \Hexpovect{1} (\omhr)$. Note that, quoting \cite{art-joyce-1}, for any $x \in \omhr$, using a change of variables $z=\Ghr(x) \in \Omega$, one has for all $i = 1, \dots, d$, $(\nabla v_{hi})^\ell(z) =  (\diff {\Ghr}){\transpose}(x)  \nabla v_{hi}^\ell{(z)},$ where~$(\diff {\Ghr}){\transpose}$ is the transpose of~$\diff \Ghr$. For simplicity, from here on now, we denote for $z \in \Omega$,
\begin{equation}
    \label{notation:G-elasticity}
    \transGhr(z) :=  (\diff {\Ghr})(x).
\end{equation}
Hence, one has,
\begin{equation}
\label{grad-elasticity}
(\na \vh)^\ell =  \big((\Gh \nabla v_{hi}^\ell )^{\transpose} \big)_{i=1}^d = \na \vhlifte \transGhr.
\end{equation}
\begin{lem}
Let $\uh, \vh \in \Hexpovect{1} (\omhr) $ and let $\uhlifte, \vhlifte \in \Hexpovect{1} (\Omega)$. Following the notation in \eqref{grad-elasticity}, we have,
\begin{multline}
\label{pass_grad_volume-Elasticity}
    \int_{\omhr} \Aomega (\euh) : \na \vh \, \dx   = \muomega \int_{\Omega} (\na\uhlifte \transGhr) : (\na\vhlifte \transGhr) \, \frac{1}{\Jhlifte} \, \dx \\
      + \muomega \int_{\Omega} (\na\uhlifte \transGhr){\transpose} : (\na\vhlifte \transGhr) \, \frac{1}{\Jhlifte} \, \dx 
      + \lambdaomega \int_{\Omega} \Tr (\na\uhlifte \transGhr)  \, \Tr(\na\vhlifte \transGhr) \, \frac{1}{\Jhlifte} \, \dx,
\end{multline}
where $\Jhlifte$ is the lift of the the Jacobian of $\Ghr$.
\end{lem}
\begin{proof}
Consider $\uh, \vh \in \Hexpovect{1} (\omhr) $ with their respective lifts $\uhlifte, \vhlifte \in \Hexpovect{1} (\Omega)$. By definition of the Hooke tensor $\Aomega$ in~\eqref{eq:Hooke-tensor}, we can seperate the integrale into two terms as follows,
\begin{equation}
\label{proof:eq-grad-vol}
    \int_{\omhr} \Aomega (\euh) : \na \vh \, \dx 
    = \muomega I_1 + \lambdaomega I_2,
\end{equation}
where $I_1  =  2 \int_{\omhr} (\euh) : (\na\vh) \, \dx$ and $I_2= \int_{\omhr}  \Tr(\euh)  \, \Tr (\na\vh) \,  \dx$. We proceed by estimating each integral separately. \medskip

For the first term, we use the expression of $\euh$ in~\eqref{eq:e(u)} as follows, 
\begin{eqnarray*}
I_1 &=& \int_{\omhr} (\na \uh) : (\na\vh) \, \dx + \int_{\omhr} (\na \uh){\transpose} : (\na\vh)  \, \dx. 
\end{eqnarray*}
Using a change of variable and Equation \eqref{grad-elasticity}, we have, 
\begin{eqnarray*}
    I_1 & = & \int_{\Omega} (\na\uhlifte \transGhr) : (\na\vhlifte \transGhr) \, \frac{1}{\Jhlifte} \, \dx  +  \int_{\Omega} (\na\uhlifte \transGhr){\transpose} : (\na\vhlifte \transGhr) \, \frac{1}{\Jhlifte} \, \dx.
\end{eqnarray*}

As for the second term, we take advantage of the fact that $\Tr(\euh) =  \Tr(\na \uh)$ and we proceed in a similar manner using a change of variable and applying Equation \eqref{grad-elasticity} as follows, 
\begin{multline*}
I_2= \int_{\omhr}  \Tr(\euh)  \, \Tr (\na\vh)  \dx
    = \int_{\omhr} \Tr(\na \uh)  \, \Tr (\na\vh) \dx 
    = \int_{\Omega} \Tr(\na\uhlifte \transGhr)  \, \Tr (\na\vhlifte \transGhr) \, \frac{1}{\Jhlifte} \, \dx.
\end{multline*}
Replacing the expression of $I_1$ and $I_2$ in \eqref{proof:eq-grad-vol} concludes the proof.
\end{proof}

\paragraph{Lift transformation estimations.}
We recall some essential properties mentioned in \cite{art-joyce-1}. Indeed, the lift transformation $\Ghr$ is globally continuous and piece-wise differentiable on each mesh element. Additionally, quoting \cite[Proposition 2]{art-joyce-1}, where the full proof is detailed: let $\tr \in \taur$, the mapping ${\Ghr}_{|_{\tr}}$ is~$\c{r+1}(\tr)$ regular and a $\c 1$- diffeomorphism from $\tr$ onto $\te$.  Moreover, for a sufficiently small mesh size $h$, there exists a constant $c>0$, independent of $h$, such that, 
\begin{equation}
\label{ineq:Gh-Id_Jh-1} 
  \forall \ x \in \tr, \ \ \ \ \| \diff {\Ghr}(x) - \I \| \le c h^r \qquad \mbox{ and } \qquad 
    | \Jh(x)- 1 | \le c h^r.
\end{equation}
where $\diff \Ghr$ is the differential of $\Ghr$ and $\Jh$ is its Jacobin. \medskip

A direct consequence of the inequalities \eqref{ineq:Gh-Id_Jh-1}, using the lift in Definition \ref{def:liftvolume}, is that both~$\transGhr$ and~$\Jhlifte$ are bounded on $\te$ independently of the mesh size h. Furthermore, the following inequalities, which are a key ingredient for the proof of the error estimations, are presented, 
    \begin{equation}
        \label{ineq:Ghr-Id_1/Jh-1}
        \forall \ x \in \te, \ \ \ \ \| \transGhr (x) - \I \| \le c h^r \qquad \mbox{ and } \qquad
         \left| \frac{1}{\Jhlifte(x)}- 1 \right| \le c h^r.
    \end{equation}
% Their full proof is presented in \cite{art-joyce-1}. 

A similar bound of $\Jblifte-1$ with respect to the mesh size $h$ and the geometrical order of the mesh $r$ is proved in~\cite{D1}: there exists a constant $c>0$ independent of $h$ such that,  
\begin{equation}
  \label{ineq:AhJh-part1}
    \left\| 1-\frac{1}{\Jblifte} \right\|_{\Linf(\Gamma)} \le ch^{r+1},
\end{equation}
\blue{where $\Jblifte$ is the lift of $\Jb$, the tangential Jacobian of $b$ the orthogonal projection on the boundary~$\Gamma$, defined \cite[Equation (2.10)]{D1}.} 

\paragraph{Estimations near the boundary.}
From this point forward, we denote $B_h^\ell \subset \Omega$ as the union of all the non-internal elements of the exact mesh $\taue$ defined in Appendix \ref{mesh:appendix}, 
\begin{equation}
    \label{eq:b_h^ell}
     B_h^\ell:=  \{ \te \in \taue ; \  \te \, \mbox{has at least two vertices on } \Gamma \}. %\ftre \circ (\ftr)^{-1}(\tr) = \te \}.
\end{equation}
Note that, by definition of $B_h^\ell$, we have,
\begin{equation}
    \label{JDG}
    \frac{1}{\Jhlifte}-1 =0  \ \ \ \mbox{and} \ \ \ \transGhr - \I = 0 \ \ \ \mbox{in} \ \Omega \backslash B_h^\ell.
\end{equation}

The following corollary involving  $B_h^\ell$ is a direct consequence of \cite[Lemma 4.10]{elliott} or~\cite[Theorem~1.5.1.10]{Grisvard2011}. It has been used in many error analysis problems such as in \cite{elliott,art-joyce-1,art-joyce-2} to gain a convergence order.
\begin{corollary}
Let $\v \in \Hexpovect{1}(\Omega)$ and $\w \in \Hexpovect{2}(\Omega)$. Then, for a sufficiently small~$h$, there exists~$c>0$ such that the following inequalities hold,
\begin{equation}
\label{h1/2_blh}
    \|\v\|_{\Lvect2(B_h^\ell)} \le c h^{1/2} \|\v\|_{\Hexpovect{1}(\Omega)} 
    \qquad \mbox{and} \qquad
    \|\w\|_{\Hexpovect{1}(B_h^\ell)} \le c h^{1/2} \|\w\|_{\Hexpovect{2}(\Omega)}.
\end{equation}
\end{corollary}

\subsection{The two main error components}

Next, are discussed the two main quantities controlling the total a priori error produced when approximating the exact solution. 

\subsubsection{The interpolation error}

We recall that $\Vhlifted$ is the lifted $\P k$ finite element space, with $k\ge1$. Its lifted interpolation operator $\Ihlifted$ of order $r$ is given by, 
\blue{\begin{equation*}
    \fonction{\Ihlifted}{[\c0 (\overline{\Omega})]^d}{\Vhlifted}{\v := (\v_i)_{i=1}^d}{\Ihlifted (\v ) := \big( \Ihlifte (\v_i) \big)_{i=1}^d,}
\end{equation*}}
where $\Ihlifte$ is the scalar lifted interpolation operator defined in \cite[\S 5.1]{art-joyce-1}.
Notice that, since $\Omega$ is an open subset of $\R^2$ or $\R^3$, then for $k\ge1$ we have the following Sobolev injection~\blue{$\Hexpo{k+1}(\Omega) \hookrightarrow \c0 (\overline{\Omega})$}. Thus, any function \blue{$\w \in \Hexpovect{k+1}(\Omega)\subset [\c0 (\overline{\Omega})]^d$} may be associated to an interpolation element~$\Ihlifted(\w) \in \Vhlifted$. 

\medskip

{We present the following interpolation inequality associated with the interpolation operator $\Ihlifted$, which plays a part in the following error estimation. }%It derives from Proposition 4.1.3 in \cite{these-J.GH}.
\begin{proposition}
\label{prop:interpolation-ineq}
Let $\v \in \Hexpovect{k+1} (\Omega)$ and $2 \le m \le k+1$. There exists a constant~$c>0$ independent of the mesh size $h$, such that the interpolation operator $\Ihlifted$ satisfies the following inequality,
$$
    \|\v-\Ihlifted \v\|_{\Lvect2(\Omega)} + h \|\v-\Ihlifted \v\|_{\Hexpovect{1}(\Omega)} \le c h^{m} \|\v\|_{\Hexpovect{m} (\Omega)}.
$$
\end{proposition}

Next, we present the continuity property of the lifted interpolation operator. 

\begin{lem}[The continuity property of the interpolation operator]
    There exists a constant $c>0$ mesh independent such that, 
    \begin{equation}
        \label{ineq:I-continue}
        \| \Ihlifted\v \|_{\Hexpovect{1}(\Omega)} \le c \|\v\|_{\Hexpovect{2} (\Omega)}, \qquad \forall \, \v \in \Hexpovect{2} (\Omega).
    \end{equation}
\end{lem}

%
%
%
% GEOM ERROR
%
%
\subsubsection{Geometric error estimation}
The geometric error, represented by the difference between $a$ and $\ahell$, is evaluated in the following proposition.
\begin{proposition}
\label{prop:err-geom}
There exists a constant $c>0$ independent of $h$, such that the following inequality holds for any $\u, \v \in \Vhlifted$,
\begin{equation}
    \label{ineq:a-al2}
     |a(\u, \v)-\ahell(\u, \v)| \le c h^r \|\na \u\|_{\Lvect2(B_h^\ell)}\|\na \v\|_{\Lvect2(B_h^\ell)} + ch^{r+1}\|\u\|_{\Lvect2 (\Gamma)}\|\v\|_{\Lvect2(\Gamma)}, %\qquad \forall \, v, w \in \Vhlifte. 
    \end{equation}
where $B_h^\ell$ is defined in \eqref{eq:b_h^ell}.
\end{proposition}

\begin{proof}
Let $\u, \v \in \Vhlifted$.  We start by recalling the detailed expression of the bilinear form~$a$ given in Proposition \ref{prop:a-express} as follows,
\begin{equation*}
a(\u,\v)  = \muomega \( \int_{\Omega} \na\u : \na\v \, \dx + \int_{\Omega} (\na\u)^{\transpose} : \na\v \, \dx \) + \lambdaomega \int_{\Omega} \div (\u)  \,\div(\v) \, \dx + \int_{\Gamma} \u\cdot \v \, \ds,
\end{equation*}

By the definitions of the bilinear forms $a$ and $\ahell$, their difference can be written as follows,
\begin{equation*}
    |a(\u, \v)-\ahell(\u, \v)| \le \muomega \, \{ a_1(\u, \v) + a_2(\u, \v) \} + \lambdaomega \, a_3(\u, \v) 
    + \, a_4(\u, \v),
\end{equation*}
where the terms $a_i$, defined on $\Vhlifted\times\Vhlifted$, are respectively given by,
\begin{align*}
     a_1(\u, \v) &:=\displaystyle \left|\int_{\Omega} \na\u : \na\v \, \dx - \int_{\Omega} (\na \u \, \transGhr) : (\na\v \, \transGhr) \, \frac{1}{\Jhlifte}\, \dx\right|, \\
    a_2(\u, \v) &:=\displaystyle \left|\int_{\Omega} (\na\u)^{\transpose} : \na\v \, \dx - \int_{\Omega} (\na \u \, \transGhr)^{\transpose} : (\na\v \, \transGhr) \, \frac{1}{\Jhlifte}\, \dx\right|, \\
     a_3(\u, \v) &:=\displaystyle \left|\int_{\Omega} \Tr (\na\u) \, \Tr( \na\v) \dx - \int_{\Omega}  \Tr (\na \u \, \transGhr) \,  \Tr(\na\v \, \transGhr) \, \frac{1}{\Jhlifte}\, \dx\right|, \\
     a_4(\u, \v) &:= \displaystyle  \left|\int_{\Gamma} \u \cdot \v \ (1- \frac{1}{\Jblifte})\, \ds\right|, 
\end{align*}
where $\transGhr$ is defined in \eqref{notation:G-elasticity}. The next step is to bound each~$a_i$, for $i=1, \dots, 4$, while using the inequalities in \eqref{ineq:Ghr-Id_1/Jh-1} and in \eqref{ineq:AhJh-part1} where we bound $\| \transGhr - \I \|$, $| \frac{1}{\Jhlifte}-1 |$ and~$| \frac{1}{\Jblifte}-1|$ with respect to~$h$ and $r$. \medskip

First of all, we break down the first term as follows, $a_1(\u, \v) \le Q_1 +Q_2+Q_3$, 
where, 
\begin{eqnarray*}
     Q_1 &:=& \displaystyle \left|\int_{\Omega}  \bigg(\na \u \, ( \transGhr -\I) \bigg) : ( \na \v \, \transGhr) \,\frac{1}{\Jhlifte}\, \dx\right|,\\
     Q_2 &:=& \displaystyle \left|\int_{\Omega} \na \u : \bigg(\na \v\, ( \transGhr -\I) \bigg) \,\frac{1}{\Jhlifte}\, \dx\right|,\\
     Q_3 &:=& \displaystyle \left|\int_{\Omega}  \na \u : \na \v \, (\frac{1}{\Jhlifte}-1) \, \dx\right|.
\end{eqnarray*}
We recall that $\transGhr - \I = 0, \, \mbox{and} \, \frac{1}{\Jhlifte}-1 =0 \ \mbox{in} \ \Omega \backslash B_h^\ell$, as mentioned in \eqref{JDG}. Taking advantage of these equations, we apply the inequalities in \eqref{ineq:Ghr-Id_1/Jh-1} to estimate each $Q_j$ as follows,

% We apply respectively Equation \eqref{JDG} and Inequalities \eqref{ineq:Ghr-Id_1/Jh-1} to estimate each $Q_j$ as follows, 
\begin{align*}
    & Q_1 = \left|\int_{B_h^\ell} \bigg(\na \u \, ( \transGhr -\I) \bigg) : ( \na \v \, \transGhr) \, \frac{1}{\Jhlifte}\, \dx\right|\le ch^r \| \nabla \u\|_{\Lvect2(B_h^\ell)} \| \nabla \v\|_{\Lvect2(B_h^\ell)},\\
    & Q_2 = \left|\int_{B_h^\ell} \na \u : \bigg(\na \v\, ( \transGhr -\I) \bigg) \, \frac{1}{\Jhlifte}\, \dx\right| \le ch^r \| \nabla \u\|_{\Lvect2(B_h^\ell)} \| \nabla \v\|_{\Lvect2(B_h^\ell)},\\
    & Q_3 = \left|\int_{B_h^\ell}  \na \u : \na \v \,(\frac{1}{\Jhlifte}-1) \, \dx \right| \le ch^r \| \nabla \u\|_{\Lvect2(B_h^\ell)} \| \nabla \v\|_{\Lvect2(B_h^\ell)}.
\end{align*}
Summing up the latter terms, we get,
$a_1(\u,\v) \le ch^r \| \nabla \u\|_{\Lvect2(B_h^\ell)} \| \nabla \v\|_{\Lvect2(B_h^\ell)}.$ \medskip

\blue{In a similar manner, we break down $a_2$ and apply \eqref{JDG} and \eqref{ineq:Ghr-Id_1/Jh-1} to obtain,
$$a_2(\u,\v) \le ch^r \| \nabla \u\|_{\Lvect2(B_h^\ell)} \| \nabla \v\|_{\Lvect2(B_h^\ell)}.$$}

\blue{Next, we decompose $a_3$ as follows, $a_3(\u, \v) \le S_1 + S_2+ S_3$, 
where, 
\begin{eqnarray*}
     S_1 &:=& \displaystyle \left|\int_{\Omega}  \Tr \bigg(\na \u \, ( \transGhr -\I) \bigg) \, \Tr ( \na \v \, \transGhr)\, \frac{1}{\Jhlifte}\, \dx\right|,\\
     S_2 &:=& \displaystyle \left|\int_{\Omega} \Tr (\na \u) \, \Tr\bigg(\na \v\, ( \transGhr -\I) \bigg) \, \frac{1}{\Jhlifte}\, \dx\right|,\\
     S_3 &:=& \displaystyle \left|\int_{\Omega} \Tr(\na \u) \, \Tr (\na \v) \, \bigg(\frac{1}{\Jhlifte}-1\bigg) \, \dx\right|.
\end{eqnarray*}
Equation \eqref{JDG} alongside the estimates $\|\transGhr - \I\|$ and $|\frac{1}{\Jhlifte}-1|$ in \eqref{ineq:Ghr-Id_1/Jh-1} are respectively applied to estimate each~$S_j$ in order to obtain,
$$a_3(\u,\v) \le ch^r \| \nabla \u\|_{\Lvect2(B_h^\ell)} \| \nabla \v\|_{\Lvect2(B_h^\ell)}.$$} \medskip

Finally, $a_4$ can be bounded simply by using \eqref{ineq:AhJh-part1}, where we bound $|\frac{1}{\Jblifte}-1|$ as follows,
\begin{align*}
   a_4(\u , \v) = \left| \int_{\Gamma}\u \cdot \v \, \bigg(1- \frac{1}{\Jblifte}\bigg)\, \ds\right|  \le ch^{r+1} \|  \u\|_{\Lvect2(\Gamma)}\|  \v\|_{\Lvect2(\Gamma)}.
\end{align*}

To conclude, Inequality \eqref{ineq:a-al2} is easy to obtain when summing up $a_i$, for all~$i=1, \dots, 4$, since~$\muomega, \lambdaomega$ are mesh independent constants.
\end{proof}

\blue{The following lemma is a direct consequence of the geometric error bound established in Proposition~\ref{prop:err-geom}, and will be essential for the forthcoming proofs of the error estimates.}

\begin{lem}
    Denote $\u$ (resp. $\uh$) the solution of \eqref{fv_faible} (resp. \eqref{fvh-EV}). Then, for a sufficiently small $h$, there exists a constant $c>0$ is independent of $h$ such that, 
\begin{equation}
    \label{rem:bound-uhell-indep-h}
     \|\uhlifte \|_{\Hexpovect{1}(\Omega) } \le c\|\u \|_{\Hexpovect{1}(\Omega)},
\end{equation} 
where $\uhlifte$ is the lift of $\uh$.
\end{lem}

\begin{proof}
Indeed, a relatively easy way to prove Inequality \eqref{rem:bound-uhell-indep-h} is by using the geometric error estimation~\eqref{ineq:a-al2}, we have,
\begin{align*}
   c_c \|\uhlifte \|_{\Hexpovect{1}(\Omega)}^2  \le a(\uhlifte,\uhlifte)
   \le a(\uhlifte,\uhlifte)-a(\u,\uhlifte) + a(\u,\uhlifte),
\end{align*}
where $c_c$ denotes the coercivity constant. Using Equality \eqref{rem:a=ahell-For-vhell}, we get,
\begin{equation*}
    c_c \|\uhlifte \|_{\Hexpovect{1}(\Omega)}^2 \le a(\uhlifte,\uhlifte)-a_h^\ell(\uhlifte,\uhlifte) + a(\u,\uhlifte)
    = (a-a_h^\ell)(\uhlifte,\uhlifte) + a(\u,\uhlifte).
\end{equation*}
Then we apply the geometric error estimation \eqref{ineq:a-al2} along with the continuity of $a$ with respect to the $\Hexpovect{1}(\Omega)$ norm as follows,
\begin{align*}
    \|\uhlifte \|_{\Hexpovect{1}(\Omega)}^2  &\le c h^r \|\na \uhlifte\|^2_{\Lvect2(B_h^\ell)}+ ch^{r+1}\|\uhlifte\|_{\Lvect2 (\Gamma)}^2+ c\|\u \|_{\Hexpovect{1}(\Omega)} \| \uhlifte\|_{\Hexpovect{1}(\Omega)} \\
    & \le c h^r \|\uhlifte\|^2_{\Hexpovect{1}(\Omega)}+ c\|\u \|_{\Hexpovect{1}(\Omega)} \| \uhlifte\|_{\Hexpovect{1}(\Omega)}.
\end{align*}
Thus, we have, 
$$
   (1-ch^r) \|\uhlifte \|_{\Hexpovect{1}(\Omega)}^2
    \le c\|\u \|_{\Hexpovect{1}(\Omega) } \| \uhlifte\|_{\Hexpovect{1}(\Omega)}.
$$
For a sufficiently small $h$, we have $1-ch^r\ge \epsilon$, for a given $\epsilon>0$, which concludes the proof.
\end{proof}
\subsection{Proof of the $\Hexpovect{1}$ error bound in Theorem \ref{th-error-bound}}
\label{sec:h1-err}
To begin with, we need to notice that the error $ \| \u-\uhlifte\|_{\Hexpovect{1}(\Omega) } $ can be separated as follows using the interpolation operator $\Ihlifted$, 
\begin{equation}
\label{proof:eH1}
     \| \u-\uhlifte\|_{\Hexpovect{1}(\Omega) } 
     \le \| \u-\Ihlifted \u\|_{\Hexpovect{1}(\Omega) } + \|\Ihlifted \u-\uhlifte\|_{\Hexpovect{1}(\Omega) }.
\end{equation}
We proceed by bounding each term separately.\medskip

The first term can be bounded using the interpolation inequality given in Proposition \ref{prop:interpolation-ineq} as follows,
\begin{equation}
\label{proof:interpolation}
\| \u-\Ihlifted \u\|_{\Hexpovect{1}(\Omega)} \le c h^k \| \u\|_{\Hexpovect{k+1}(\Omega)}.
\end{equation}

Afterwards, we proceed by bounding the remaining term.  Since the bilinear form $a$ is coercive with respect to the norm of $\Hexpovect{1}(\Omega)$, denoting~$c_c$ as the coercivity constant, we have,
\begin{multline*}
    c_c \|\Ihlifted \u -\uhlifte\|_{\Hexpovect{1}(\Omega)}^2 \le a(\Ihlifted \u -\uhlifte,\Ihlifted \u -\uhlifte) 
    = a(\Ihlifted \u ,\Ihlifted \u -\uhlifte) -a(\uhlifte,\Ihlifted \u -\uhlifte)   \\
     = a(\Ihlifted \u -\u ,\Ihlifted \u -\uhlifte) + a( \u ,\Ihlifted \u -\uhlifte) - a(\uhlifte,\Ihlifted \u -\uhlifte),
%\end{split}
%\end{equation*}
\end{multline*}
where in the latter equation, we added and subtracted $a( \u ,\Ihlifted \u -\uhlifte)$. Afterwards, we apply Equation \eqref{rem:a=ahell-For-vhell}, for $\v=\Ihlifted \u -\uhlifte \in \Vhlifted$, 
\begin{equation*}
   c_c \|\Ihlifted \u -\uhlifte\|_{\Hexpovect{1}(\Omega)}^2 \le |a(\Ihlifted \u -\u ,\Ihlifted \u -\uhlifte)| + \big| \ahell-a\big|(\uhlifte,\Ihlifted \u -\uhlifte).
\end{equation*} 
Taking advantage of the continuity of $a$ and the geometric estimate \eqref{ineq:a-al2}, we obtain,
\begin{equation*}
\begin{array}{l}
 c_c \|\Ihlifted \u -\uhlifte\|_{\Hexpovect{1}(\Omega)}^2\\[0.1cm]
 \begin{array}{rcl}
 & \le & \!\!\!\displaystyle c \big(  h^r \|\na \uhlifte\|_{\Lvect2(B_h^\ell)}\|\na (\Ihlifted \u-\uhlifte)\|_{\Lvect2(B_h^\ell)}  + h^{r+1}\|\uhlifte\|_{\Lvect2(\Gamma)}\|\Ihlifted \u-\uhlifte\|_{\Lvect2(\Gamma)}\big) \!\!\! \\[0.1cm]
  & & \displaystyle  \qquad{ }  + c_{cont} \|\Ihlifted \u -\u\|_{\Hexpovect{1}(\Omega)} \|\Ihlifted \u -\uhlifte\|_{\Hexpovect{1}(\Omega)} \\[0.1cm]
  & \le &\!\!\! \displaystyle  c \big( h^r \|\na \uhlifte\|_{\Lvect2(B_h^\ell)} + h^{r+1}\|\uhlifte\|_{\Lvect2(\Gamma)}  +  \|\Ihlifted \u -\u\|_{\Hexpovect{1}(\Omega)} \big) \|\Ihlifted \u -\uhlifte\|_{\Hexpovect{1}(\Omega)}.
\end{array}
\end{array}
\end{equation*}
Then, dividing by $\|\Ihlifted \u -\uhlifte\|_{\Hexpovect{1}(\Omega)}$, we have, 
\begin{align}
\label{proof:second-term}
    \|\Ihlifted \u -\uhlifte\|_{\Hexpovect{1}(\Omega)} & \le c \big( h^r \|\na \uhlifte\|_{\Lvect2(B_h^\ell)} + h^{r+1}\|\uhlifte\|_{\Lvect2(\Gamma)}  +  \|\Ihlifted \u -\u\|_{\Hexpovect{1}(\Omega)} \big).
\end{align}

To conclude, we replace Inequality \eqref{proof:second-term} in the error estimation \eqref{proof:eH1} as follows, 
\begin{equation*}
     \| \u-\uhlifte\|_{\Hexpovect{1}(\Omega)} 
       \le c \big( h^r \|\na \uhlifte\|_{\Lvect2(B_h^\ell)} + h^{r+1}\|\uhlifte\|_{\Lvect2(\Gamma)}  + \|\u-\Ihlifted \u\|_{\Hexpovect{1}(\Omega)} \big).  
\end{equation*}
Lastly, applying \eqref{proof:interpolation}, we arrive at, 
\begin{equation*}
\begin{array}{l}
\| \u-\uhlifte\|_{\Hexpovect{1}(\Omega)}\\[0.1cm]
      \le  c \big( h^r \|\na \uhlifte\|_{\Lvect2(B_h^\ell)} + h^{r+1}\|\uhlifte\|_{\Lvect2(\Gamma)}  +h^k\|\u\|_{\Hexpovect{k+1}(\Omega)} \big) \\[0.1cm]
      \le  c h^r (\|\na (\u-\uhlifte)\|_{\Lvect2(B_h^\ell)}+  \| \na \u\|_{\Lvect2(B_h^\ell)}) + ch^{r+1}\|\uhlifte\|_{\Lvect2(\Gamma)}  +ch^k\|\u\|_{\Hexpovect{k+1}(\Omega)} \\[0.1cm]
      \le c h^r (\|\u-\uhlifte\|_{\Hexpovect{1}(\Omega)}+ h^{1/2} \|\u\|_{\Hexpovect{2}(\Omega)}) + ch^{r+1}\|\uhlifte\|_{\Lvect2(\Gamma)}  +ch^k\|\u\|_{\Hexpovect{k+1}(\Omega)},
\end{array}
\end{equation*}
where we also used \eqref{h1/2_blh} to gain $h^{1/2} \|\u\|_{\Hexpovect{2}(\Omega)}$. 
Thus we have,
\begin{align*}
    (1-c h^r) \|\u-\uhlifte\|_{\Hexpovect{1}(\Omega)} & \le c \left( h^{r+1/2}\|\u\|_{\Hexpovect{2}(\Omega)} + h^{r+1}\|\uhlifte\|_{\Lvect2(\Gamma)}  +h^k\|\u\|_{\Hexpovect{k+1}(\Omega)} \right). 
\end{align*}
For a sufficiently small $h$, we arrive at,
\begin{align*}
    \|\u-\uhlifte\|_{\Hexpovect{1}(\Omega)} & \le c \left( h^{r+1/2}\|\u\|_{\Hexpovect{2}(\Omega)} + h^{r+1}\|\uhlifte\|_{\Lvect2(\Gamma)}  +h^k\|\u\|_{\Hexpovect{k+1}(\Omega)} \right).
\end{align*}
This provides the desired result using Inequality \eqref{rem:bound-uhell-indep-h}.

\subsection{Proof of the $\Lvect2$ error bound in Theorem \ref{th-error-bound}}
\label{sec:L2-err}

% Recall that $\u \in \H$ denotes the solution of the variational problem \eqref{fv_faible},~$\uh \in \Vhd$ denotes the solution of the discrete problem \eqref{fvh-EV}, and $\uhlifte$ is its lift onto $\Omega$, given in Definition \eqref{def:liftvolume-Elasticity}. 
To estimate the $\Lvect2$ norm of the error, we define the functional~$F_h$ by,
$$
\fonction{F_h }{\Hexpovect{1}(\Omega)}{\R}{\v}{\displaystyle   F_h(\v)=a(\u-\uhlifte,\v).}
$$
\blue{We begin by bounding $|F_h(\v)|$ for any $\v \in \Hexpovect{2} (\Omega)$ in Lemma \ref{lem:Fh}. Afterwards an Aubin-Nitsche argument is applied in order to prove the $\Lvect2$ error estimation \eqref{errh1_errl2}, following the same approach as in the scalar case discussed in \cite{art-joyce-1}. For the sake of clarity, we point out that the main challenge in this framework consists in ensuring that the various estimates continue to hold in the context of linear elasticity with their adequate norms involving vector-valued functions.}
\begin{lem}
\label{lem:Fh}
For a sufficiently small~$h$, there exists a mesh independent constant~$c>0$ such that the following inequality holds for any $\v \in \Hexpovect{2}(\Omega)$,
\begin{equation}
\label{ineq:F-h}
     |F_h(\v)| \le c ( h^{k+1} + h^{r+1} ) \|\v\|_{\Hexpovect{2}(\Omega)}. 
\end{equation}
\end{lem}
We start by summarizing the essential ingredients to prove this lemma.
% Firstly, 
% Inequality \eqref{h1/2_blh-elasticity} states that,
% \begin{equation}
% \label{ineq:blh-for-v}
%    \|\na \v\|_{\L2(B_h^\ell)} \le c h^{1/2} \|\v\|_{\Hexpovect{2}(\Omega)}, \quad \forall \, \v \in \Hexpovect{2}\omgam.
% \end{equation}Secondly, 
The interpolation inequality in Proposition \ref{prop:interpolation-ineq} implies that, 
\begin{equation}
    \label{ineq:interpolation-v}
    \| \Ihlifted \v - \v \|_{\Hexpovect{1} (\Omega)} \le c h \|\v\|_{\Hexpovect{2} (\Omega)}, \quad \forall \, \v \in \Hexpovect{2}(\Omega).
\end{equation}
Moreover, applying Equality \eqref{rem:a=ahell-For-vhell} for $\Ihlifted \v \in \Vhlifted$, we have, 
\begin{equation}
    \label{eq:a=ahell-For-Iv}
    a(\u,\Ihlifted \v) = l(\Ihlifted \v) = \ahell (\uhlifte, \Ihlifted \v).
\end{equation}
\begin{proof}[Proof of Lemma \ref{lem:Fh}]%
Consider $\v \in \Hexpovect{2}(\Omega)$. To begin with, we decompose $|F_h(\v)|$ in two terms as follows,
\begin{multline*}
    |F_h(\v)|  = |a(\u-\uhlifte, \v)| 
           = |a(\u-\uhlifte, \v) +a(\u-\uhlifte, \Ihlifted \v )-a(\u-\uhlifte, \Ihlifted \v)| \\
          \le |a(\u-\uhlifte, \v- \Ihlifted \v)| + | a(\u-\uhlifte, \Ihlifted \v)| 
          =: F_1 + F_2.
\end{multline*}

Firstly, to bound $F_1$, we take advantage of the continuity of the bilinear form $a$ with respect to the norm $\|\cdot\|_{\Hexpovect{1}(\Omega)}$ and apply the $\Hexpovect{1}(\Omega)$ error estimation \eqref{errh1_errl2} as follows,  
\begin{multline*}
    F_1  %= |a(u-u_h^\ell, v- \Ihlifte v)| 
         \le c_{cont} \, \| \u-\uhlifte\|_{\Hexpovect{1}(\Omega)} \| \v-\Ihlifted \v \|_{\Hexpovect{1}(\Omega)}  
         \le c ( h^k + h^{r+1/2} ) \, h \|\v\|_{\Hexpovect{2} (\Omega)} 
         \le c ( h^{k+1} + h^{r+3/2} ) \, \|\v\|_{\Hexpovect{2} (\Omega)},
\end{multline*}
where we used Inequality~\eqref{ineq:interpolation-v}. \medskip

Secondly, to estimate $F_2$, we apply that Equality \eqref{eq:a=ahell-For-Iv} with the geometric error estimation \eqref{ineq:a-al2} as follows,
\begin{multline*}
      F_2 = |a(\u, \Ihlifted \v)-a(\uhlifte, \Ihlifted \v)| 
         = |\ahell(\uhlifte, \Ihlifted \v)-a(\uhlifte, \Ihlifted \v)| 
         = |(\ahell-a)(\uhlifte, \Ihlifted \v)| \\
         \le c h^r \|\na \uhlifte\|_{\Lvect2(B_h^\ell)}\|\na (\Ihlifted \v)\|_{\Lvect2(B_h^\ell)} +c h^{r+1}\|\uhlifte \|_{\Lvect2(\Gamma)} \| \Ihlifted \v\|_{\Lvect2(\Gamma)}.
\end{multline*} 
Next, we will bound the first term in the latter inequality separately, as follows,
\begin{eqnarray*}
   F_3 & := & h^r \|\na \uhlifte\|_{\Lvect2(B_h^\ell)}\|\na (\Ihlifted \v)\|_{\Lvect2(B_h^\ell)} \\
       & \le & h^r \Big( \|\na (\uhlifte-\u) \|_{\Lvect2(B_h^\ell)} + \|\na \u\|_{\Lvect2(B_h^\ell)} \Big) \Big( \|\na (\Ihlifted \v - \v)\|_{\Lvect2(B_h^\ell)}+\|\na \v\|_{\Lvect2(B_h^\ell)} \Big)\\
       & \le & h^r \Big( \| \uhlifte-\u \|_{\Hexpovect{1} (\Omega)} + \|\na \u\|_{\Lvect2(B_h^\ell)} \Big) \Big( \|\Ihlifted \v-\v\|_{\Hexpovect{1} (\Omega)}+\|\na \v\|_{\Lvect2(B_h^\ell)} \Big).
\end{eqnarray*}
We now apply respectively the $\Hexpovect{1}(\Omega)$ error estimation \eqref{errh1_errl2}, Inequality~\eqref{h1/2_blh} and the interpolation inequality \eqref{ineq:interpolation-v}, as follows,
\begin{eqnarray*}
   F_3 & \le & c \, h^r \Big( h^k + h^{r+1/2} + h^{1/2} \| \u\|_{\Hexpovect{2}(\Omega)} \Big) \Big( h\|\v\|_{\Hexpovect{2} (\Omega)}+h^{1/2}\| \v\|_{\Hexpovect{2}(\Omega)} \Big)\\
   & \le & c \, h^r \, h^{1/2} \Big( h^{k-1/2} + h^{r} + \| \u\|_{\Hexpovect{2}(\Omega)} \Big) \Big( h^{1/2}+1\Big)\, h^{1/2} \| \v\|_{\Hexpovect{2} (\Omega)} \\
    & \le & c \, h^{r+1} \Big( h^{k-1/2} + h^{r} + \| \u\|_{\Hexpovect{2}(\Omega)} \Big) \Big( h^{1/2}+1\Big)\, \| \v\|_{\Hexpovect{2} (\Omega)}.
\end{eqnarray*}
Noticing that $k-1/2>0$ (since $k\ge 1$) and that $\Big( h^{k-1/2} + h^{r} + \| \u\|_{\Hexpovect{2}(\Omega)} \Big) \Big( h^{1/2}+1\Big)$ is bounded by a constant independent of $h$, we obtain,
\begin{equation*}
   F_3 \le c \, h^{r+1} \| \v\|_{\Hexpovect{2} (\Omega)}.
\end{equation*}
Replacing the latter estimation in the expression of $F_2$ and {by the trace inequality}, we have, 
$$
    F_2  \le  c  h^{r+1} \|\v\|_{\Hexpovect{2} (\Omega)} +c h^{r+1}\|\uhlifte \|_{\Hexpovect{1}(\Omega)} \| \Ihlifted \v\|_{\Hexpovect{1}(\Omega)}.
$$
Moreover, using Inequality \eqref{ineq:I-continue}, which states that $\| \Ihlifted \v\|_{\Hexpovect{1}(\Omega)} \le c \| \v\|_{\Hexpovect2(\Omega)}$, and by applying Inequality \eqref{rem:bound-uhell-indep-h}, we get,
$$
   F_2   \le  c  h^{r+1} \| \v\|_{\Hexpovect2(\Omega)} +c h^{r+1}\|\u \|_{\Hexpovect{1}(\Omega)}\| \v\|_{\Hexpovect2(\Omega)} 
   \le  c h^{r+1} \| \v\|_{\Hexpovect2(\Omega)}.
$$

We conclude the proof by summing the estimates of $F_1$ and~$F_2$.
\end{proof}
We can now prove the $\Lvect2$ estimate \eqref{errh1_errl2}, using an Aubin–Nitsche duality argument.
\begin{proof}[\textbf{Proof of the $\Lvect2$ estimate \eqref{errh1_errl2}.}]
Defining $\e := \u-\uhlifte \in \Hexpovect{1}(\Omega)$, we aim  to estimate the $\Lvect2$ error norm: $\|\e\|_{\Lvect2\omgam}^2 = \|\u-\uhlifte\|_{\Lvect2(\Omega)}^2 + \|\u-\uhlifte\|_{\Lvect2(\Gamma)}^2.$
In order to do that, an Aubin–Nitsche duality argument is used. We apply Theorem~\ref{th_existance_unicite_u} for $\f=\e$ and~$\g= \e_{|_{\Gamma}}$ as follows: there exists a unique solution $\mathbf{z_{e}} \in \Hexpovect{2}(\Omega)$ to Problem \eqref{fv_faible}. By the regularity assumptions considered, $\mathbf{z_{e}}$ satisfies Inequality \eqref{ineq:energy} as follows,
\begin{equation}
\label{reg}
    \|\mathbf{z_e}\|_{\Hexpovect{2}(\Omega)} \le c \|\e\|_{\Lvect2\omgam}.
\end{equation}
Notice that, 
$$
    \|\e\|_{\Lvect2\omgam}^2=a(\e,\mathbf{z_e})=|F_h(\mathbf{z_e})|.
$$
Applying Inequality \eqref{ineq:F-h} and Inequality \eqref{reg}, we have, \blue{
$$
 \|\e\|_{\Lvect2\omgam}^2\le c ( h^{k+1} + h^{r+1} ) \|\mathbf{z_e}\|_{\Hexpovect2(\Omega)} \le c ( h^{k+1} + h^{r+1} ) \|\e\|_{\Lvect2\omgam},
$$}
which concludes the proof.

\end{proof}

\appendix

\section{Affine and exact mesh definition}
\label{mesh:appendix}

The constructions of the mesh used and of the lift procedure presented in Section \ref{sec:mesh-lift} are based on the following fundamental results that may be found in \cite{tubneig} and \cite[\S 14.6]{GT98}. For more details, we refer to \cite[Ch. 3]{these-J.GH}.
%
%
%
%
% Tub_neigh_prop
%
%
%
%
\begin{proposition}
\label{tub_neigh_orth_proj_prop}
Let $\Omega$ be a nonempty bounded connected open subset of $\R^{d}$ %$(d~=~2,3)$ 
with a~$\c2$ boundary~$\Gamma= \partial \Omega$. Let $\d : \R^d \to \R$ be the signed distance function with respect to $\Gamma$ defined by,
\begin{displaymath}
  \d(x) :=
   \{
    \begin{array}{ll}
      -\dist(x, \Gamma)&  {\rm if } \, x \in \Omega ,
      \\
      0&  {\rm if } \, x \in \Gamma ,
      \\
      \dist(x, \Gamma)&  {\rm otherwise},
    \end{array}
  \right. \qquad 
  {\rm with} \quad \dist(x, \Gamma) := \inf \{|x-y|,~ \ y \in \Gamma \}.
\end{displaymath}
Then there exists a tubular neighborhood $\mathcal{U}_{\Gamma}:= \{ x \in \R^d ; |\d(x)| < \delta_\Gamma \}$ of $\Gamma$, of sufficiently small width $\delta_\Gamma$, where {$\d$ is a $\c2$ function}. Its gradient 
$\na \d$ is an extension of the external unit normal~$\nn$ to $\Gamma$. Additionally, in this neighborhood~$\mathcal{U}_{\Gamma}$, the orthogonal projection~$b$ onto $\Gamma$ is uniquely defined and given by,
\begin{displaymath} 
b\, :~ x \in \mathcal{U}_{\Gamma}  
\longmapsto    b(x):=x-\d(x)\na \d(x) \in \Gamma.
\end{displaymath}
\end{proposition}
\subsection{Affine mesh $\tauh$} 
%\paragraph{Affine mesh $\tauh$} 
Let $\tauh$ be a polyhedral mesh of $\Omega$ made of simplices of dimension $d$ (triangles or tetrahedra), it is chosen as quasi-uniform and henceforth shape-regular (see  \cite[definition 4.4.13]{quasi-unif}). Define the mesh size~$h:= \max\{\mathrm{diam}(T); T \in \tauh \}$, where $\mathrm{diam}(T)$ is the diameter of $T$. The mesh domain is denoted by $\omh1:= \cup_{T\in  \tauh}T$. Its boundary denoted by $\Gamma_h^{(1)} :=\partial \omh1$ is composed of $(d-1)$-dimensional simplices that form a  mesh of $\Gamma = \partial \Omega$. The vertices of $\Gamma_h^{(1)}$ are assumed to lie on $\Gamma$. For $T \in \tauh$, we define an affine function that maps the reference element onto~$T$,~$\ft : \tref \to T:=\ft(\tref).$
For more details, see \cite[page 239]{ciaravtransf}. 
\subsection{Exact mesh $\taue$}
% \label{sec:exact-mesh}
In this section, is recalled the definition of an exact transformation $\fte$ defined in the work of Elliott {\it et al.} in~\cite{elliott} in 2013, which is used throughout this work. For the sake of completeness, one needs to recall that in the 1970's, Scott gave an explicit construction of an exact triangulation in two dimensions in \cite{scott-2}. Later on, it was generalised by Lenoir in \cite{Lenoir1986}. The present definition of an exact transformation~$\fte$ combines the definitions found in \cite{Lenoir1986,scott-2,nedelec,Bernardi1989} with the orthogonal projection onto the domain's boundary $b$, defined in Proposition \ref{tub_neigh_orth_proj_prop}, first used to this aim by Dubois in \cite{dubois} in the 1990's. 

\medskip

Under the assumption of a quasi-uniformal mesh and for a sufficiently small $h$, a mesh element~$T\in \tauh$ cannot have $d+1$ vertices on the boundary $\Gamma$.  In Definition \ref{appdef:sigma-lambdaetoile-haty}, are given essential key elements for the construction of $\fte$.
\begin{definition}
\label{appdef:sigma-lambdaetoile-haty}
    Let $T\in\tauh$ be a non-internal element (having at least $2$ vertices on the boundary). Denote~$v_i = \ft(\hatv_i)$ as its vertices, where $\hatv_i$ are the vertices of~$\tref$. We define $\varepsilon_i=1$  if~$v_i\in  \Gamma$ and~$\varepsilon_i=0$ otherwise.
    To $\hatx\in \tref$ is associated its barycentric coordinates~$\lambda_i$ associated to the vertices $\hatv_i$ of~$\tref$ and $\lambdaetoile (\hatx):= \sum_{i=1}^{d+1} \varepsilon_i \lambda_i$ (shortly denoted by $\lambdaetoile$). Finally, we define  $\hat{\sigma} : = \{ \hatx \in \tref; \lambdaetoile(\hatx) = 0 \}$ and the function~$
    \haty := \dfrac{1}{\lambdaetoile}\sum_{i=1}^{d+1} \varepsilon_i \lambda_i\hatv_i\in\tref$, which is well defined on $\trefminissigma$.
\end{definition}
\begin{definition}
% \label{def:fte-y}
We denote~$\taue$ the mesh consisting of all exact elements $\te=\fte(\tref)$, where $\fte = \ft$ for all internal elements of $\tauh$, as for the case of non-internal elements $\fte$ is given by, 
\begin{equation*}
  % \label{eq:def-fte}
\fonction{\fte}{\tref  }{\te :=\fte( \tref) }{  \hatx}{\displaystyle  \fte( \hatx) := \left\lbrace 
\begin{array}{ll}
 x & {\rm if } \, \hatx \in \hatsigma, \\
     x+(\lambdaetoile)^{r+2} ( b(y) - y) &  {\rm if } \, \hatx \in \trefminissigma ,
\end{array}
\right.}
\end{equation*}
with $x = \ft( \hatx)$ and $y = \ft( \haty)$ and for an integer $r \ge 1$, the value of which is discussed in the following remark. 
% It has been proven in \cite{elliott} that $\fte$ is a $\c1$-diffeomorphism and $C^{r+1}$ regular on $\tref$.  
\end{definition}

\begin{remark}[$\fte$ regularity]
It has been proven in \cite{elliott} that the exact transformation $\fte$ is a $\c1$-diffeomorphism and $C^{r+1}$ regular on $\tref$. Indeed for any $\hatx \in \trefminissigma$, the function $\fte(\hatx)=x+(\lambdaetoile)^{s} ( b(y) - y)$ has an exponent $s=r+2$ inherited from \cite{elliott}: this exponent value guaranties the~$\c{r+1}$ regularity of the function $\fte$.
\end{remark}

\section{The lift transformation definition}
\label{appendix:lift}
We recall that the idea of lifting a function from the discrete domain onto the continuous one was already treated and discussed in many articles dating back to the 1970's, like \cite{nedelec,scott,Lenoir1986,Bernardi1989}. The key ingredient is a well defined lift transformation going from the mesh domain onto the physical domain $\Omega$. \medskip

We recall the lift transformation $\Ghr$, which was defined in \cite[\S 4]{art-joyce-1}. Following the notations given in Definition \ref{appdef:sigma-lambdaetoile-haty}, we introduce the transformation $\Ghr:\omhr \rightarrow \Omega$ given piecewise for all~$\tr\in \taur$ by,
\begin{equation*}
% \label{eq:def-fter}
    {\Ghr}_{|_{\tr}} := \ftre \circ ({\ftr})^{-1},
\end{equation*}
where the transformation $\ftre$ is given as follows, for $\hatx \in \tref$,
\begin{equation*}
% \label{eq:def-fter-2}
  \ftre( \hatx) \!:=\! \left\lbrace 
\begin{array}{ll}
 \!\!   x & {\rm if } \, \hatx \in \hatsigma \\
  \!\!   x+(\lambdaetoile)^{r+2} ( b(y) - y) &  {\rm if } \, \hatx \in \trefminissigma
\end{array}
\right.,
\end{equation*}
with $ x := \ftr( \hatx)$ and $y := \ftr( \haty)$, where $\ftr$ is the polynomial transformation defined in Section~\ref{sec:mesh-lift}. Notice that this definition implies that~${\Ghr}_{|_{\tr}} = id_{|_{\tr}}$, for any internal mesh element~$\tr \in \taur$. Note that, by construction, $\Ghr$ is globally continuous and piecewise differentiable on each mesh element. \medskip

For the sake of completeness, in the following example we illustrate the transformation $\Ghr$ on a quadratic mesh ($r=2$).
\begin{ex} 
We display in this example the effect of $\Ghr$ on the elements of the curved mesh~$\taur$, for $r=2$. In Figure \ref{fig:Gh2}, we display the transformation $\Ghdeux$ that maps a curved element $\tdeux \in \taudeux$ into an exact element $\te$. 
\begin{figure}[h]
\centering
\begin{tikzpicture}
%%%%%%%%%%% curved 

\draw[red] (0.4,0.8) node[above] {$\tdeux$};
\draw[red] (-1,0) arc (180:123:3);

%%%%%%%%%%%%%%%%%%%%%%%%%%%% image triangle

\draw (-1,0) node  {$\bullet$};
\draw (0.36,2.5) node  {$\bullet$};
\draw (2,0) node  {$\bullet$}; 
\draw (0.36,2.56) -- (2,0);
\draw (2,0) -- (-1,0);

%%%%%%%%%%%%%%%%%%%%%%%%%%% curved triangle

\draw[red] (0.5,0) node  {$\bullet$};
\draw[red] (-0.6,1.5) node  {$\bullet$}; 
\draw[red] (1.15,1.3) node  {$\bullet$};

%%%%%%%%%%%%%%%%%%%%%% GAMMA NOT EXACT

\draw plot [domain=-0.3:1.5] (\x-1,-1.786*\x^2+4.3*\x);

\draw (-1.7,-1) node {$\Gamma$};

%%%%%%%%%%%%%%%%%% transformation %%%%%%%%%%%%%

\draw [->] (2.6,1) -- (5.2,1);
\draw (4,1) node[above] {$\Ghdeux= \textcolor{red}{\ftre} \circ \textcolor{red}{(\ftdeux)^{-1}} $};
% \draw (4,0.7) node {${= Id + \rhotr} $};

%%%%%%%%%%%%%%%%%% second triangle %%%%%%%%%%%%

%%%%%%%%%%%%%%%%%%%%%% GAMMA NOT EXACT

\draw plot [domain=-0.3:1.5] (\x+6,-1.786*\x^2+4.3*\x);

\draw (5.3,-1) node {$\Gamma$};

%%%%%%%%%%%%%%%%%%%% triangle construction 

%%%%%%%%%%% curved 
\draw[red] (7.4,0.8) node[above] {$\te$};

%%%%%%%%%%%%%%%%%%%%%%%%%%%% image triangle

\draw (6,0) node  {$\bullet$};
\draw (7.36,2.5) node  {$\bullet$};
\draw (9,0) node  {$\bullet$}; 
\draw (7.36,2.56) -- (9,0);
\draw (9,0) -- (6,0);

%%%%%%%%%%%%%%%%%%%%%%%%%%% curved triangle

\draw[red] (7.5,0) node  {$\bullet$};
\draw[red] (6.4,1.5) node  {$\bullet$}; %(6.4,1.4) pour l'autre ex
\draw[red] (8.15,1.3) node  {$\bullet$};
\draw[red] [->] (2.5,-2.5) -- (0.9,-1);
\draw[red] (0.5,-2.2) node[above] {${\ftdeux} $};
\draw[red] [->] (5.5,-2.5) -- (6.7,-1);
\draw[red] (7.3,-2.2) node[above] {${\ftre}$ }; % = \ftr + \Tilde{\varphi} $};
\draw[red] (4.4,-3.3)  node[left] {$\tref$};
\draw (3.5,-2) node  {$\bullet$};
\draw (3.5,-4) node  {$\bullet$}; 
\draw (5.5,-4) node  {$\bullet$};
\draw[red] (3.5,-3) node  {$\bullet$};
\draw[red] (4.5,-4) node  {$\bullet$}; 
\draw[red] (4.5,-3) node  {$\bullet$};
\draw (3.5,-2) -- (3.5,-4);
\draw (5.5,-4) -- (3.5,-2);
\draw (5.5,-4) -- (3.5,-4);
\end{tikzpicture}
\caption{Visualisation of $G_h^{(2)}: \tdeux \to \te$ in a 2D case, in the quadratic case ($r=2$).}
\label{fig:Gh2}
\end{figure}
\end{ex}

\bibliographystyle{abbrv}
\bibliography{biblio}

\begin{thebibliography}{10}

\bibitem{barrios2011posteriori}
T.~P. Barrios, E.~M. Behrens, and M.~A. S{\'a}nchez.
\newblock A posteriori error analysis of an augmented mixed formulation in linear elasticity with mixed and dirichlet boundary conditions.
\newblock {\em Computer Methods in Applied Mechanics and Engineering}, 200(1-4):101--113, 2011.

\bibitem{Bernardi1989}
C.~Bernardi.
\newblock Optimal finite-element interpolation on curved domains.
\newblock {\em SIAM J. Numer. Anal.}, 26(5):1212--1240, 1989.

\bibitem{D3}
A.~Bonito and A.~Demlow.
\newblock Convergence and optimality of higher-order adaptive finite element methods for eigenvalue clusters.
\newblock {\em SIAM J. Numer. Anal.}, 54(4):2379--2388, 2016.

\bibitem{D4}
A.~Bonito, A.~Demlow, and J.~Owen.
\newblock A priori error estimates for finite element approximations to eigenvalues and eigenfunctions of the {L}aplace-{B}eltrami operator.
\newblock {\em SIAM J. Numer. Anal.}, 56(5):2963--2988, 2018.

\bibitem{Gvial}
V.~Bonnaillie-No\"{e}l, D.~Brancherie, M.~Dambrine, F.~H\'{e}rau, S.~Tordeux, and G.~Vial.
\newblock Multiscale expansion and numerical approximation for surface defects.
\newblock In {\em C{ANUM} 2010, {$40^{\rm e}$} {C}ongr\`es {N}ational d'{A}nalyse {N}um\'{e}rique}, volume~33 of {\em ESAIM Proc.}, pages 22--35. EDP Sci., Les Ulis, 2011.

\bibitem{quasi-unif}
S.~C. Brenner and L.~R. Scott.
\newblock The mathematical theory of finite element methods.
\newblock 15:16,361, 2002.

\bibitem{carstensen1998posteriori}
C.~Carstensen and G.~Dolzmann.
\newblock A posteriori error estimates for mixed fem in elasticity.
\newblock {\em Numerische Mathematik}, 81(2):187--209, 1998.

\bibitem{Jaca}
F.~Caubet, J.~Ghantous, and C.~Pierre.
\newblock Numerical study of a diffusion equation with ventcel boundary condition using curved meshes.
\newblock {\em Monografías Matemáticas García de Galdeano}, 2023.

\bibitem{art-joyce-2}
F.~Caubet, J.~Ghantous, and C.~Pierre.
\newblock Finite element analysis of a spectral problem on curved meshes occurring in diffusion with high order boundary conditions.
\newblock {\em (submitted)}, 2024.

\bibitem{art-joyce-1}
F.~Caubet, J.~Ghantous, and C.~Pierre.
\newblock A priori error estimates of a poisson equation with ventcel boundary conditions on curved meshes.
\newblock {\em SIAM J. on Numer. Anal.}, 62(4):1929--1955, 2024.

\bibitem{ciarlet2022mathematical}
P.~Ciarlet.
\newblock {\em Mathematical Elasticity, Volume I: Three-Dimensional Elasticity}.
\newblock Classics in applied mathematics. Society for Industrial and Applied Mathematics, 2022.

\bibitem{ciarlet-elasticity-volume-3}
P.~G. Ciarlet.
\newblock Mathematical elasticity, vol iii, theory of shells.
\newblock 2000.

\bibitem{PHcia}
P.~G. Ciarlet.
\newblock {\em The finite element method for elliptic problems}, volume~40 of {\em Classics in Applied Mathematics}.
\newblock Society for Industrial and Applied Mathematics (SIAM), Philadelphia, PA, 2002.

\bibitem{ciaravtransf}
P.~G. Ciarlet and P.-A. Raviart.
\newblock Interpolation theory over curved elements, with applications to finite element methods.
\newblock {\em Comp. Meth. Appl. Mech. Eng.}, 1:217--249, 1972.

\bibitem{tubneig}
C.~Dapogny and P.~Frey.
\newblock Computation of the signed distance function to a discrete contour on adapted triangulation.
\newblock {\em Calcolo}, 49(3):193--219, 2012.

\bibitem{D1}
A.~Demlow.
\newblock Higher-order finite element methods and pointwise error estimates for elliptic problems on surfaces.
\newblock {\em SIAM J. Numer. Anal.}, 47(2):805--827, 2009.

\bibitem{D2}
A.~Demlow and G.~Dziuk.
\newblock An adaptive finite element method for the {L}aplace-{B}eltrami operator on implicitly defined surfaces.
\newblock {\em SIAM J. Numer. Anal.}, 45(1):421--442, 2007.

\bibitem{dubois}
F.~Dubois.
\newblock Discrete vector potential representation of a divergence-free vector field in three-dimensional domains: numerical analysis of a model problem.
\newblock {\em SIAM J. Numer. Anal.}, 27(5):1103--1141, 1990.

\bibitem{duvant2012inequalities}
G.~Duvaut and J.~L. Lions.
\newblock {\em Inequalities in mechanics and physics}, volume 219.
\newblock Springer Science \& Business Media, 2012.

\bibitem{Dz88}
G.~Dziuk.
\newblock Finite elements for the {B}eltrami operator on arbitrary surfaces.
\newblock In {\em Partial differential equations and calculus of variations}, volume 1357 of {\em Lecture Notes in Math.}, pages 142--155. Springer, Berlin, 1988.

\bibitem{ed}
D.~Edelmann.
\newblock Isoparametric finite element analysis of a generalized {R}obin boundary value problem on curved domains.
\newblock {\em SMAI J. Comput. Math.}, 7:57--73, 2021.

\bibitem{elliott}
C.~M. Elliott and T.~Ranner.
\newblock Finite element analysis for a coupled bulk-surface partial differential equation.
\newblock {\em IMA J. Numer. Anal.}, 33(2):377--402, 2013.

\bibitem{EG}
A.~Ern and J.-L. Guermond.
\newblock {\em Theory and practice of finite elements}, volume 159 of {\em Applied Mathematical Sciences}.
\newblock Springer-Verlag, New York, 2004.

\bibitem{elasticity-1}
K.~Feng and Z.-C. Shi.
\newblock {\em Mathematical theory of elastic structures}.
\newblock Springer-Verlag, Berlin; Science Press Beijing, Beijing, 1996.

\bibitem{gatica2006priori}
G.~N. Gatica and L.~F. Gatica.
\newblock On the a priori and a posteriori error analysis of a two-fold saddle-point approach for nonlinear incompressible elasticity.
\newblock {\em International journal for numerical methods in engineering}, 68(8):861--892, 2006.

\bibitem{these-J.GH}
J.~Ghantous.
\newblock {\em Consideration of high-order boundary conditions and numerical analysis of diffusion problems on curved meshes using high-order finite elements}.
\newblock PhD thesis, Universit{\'e} de Pau et des Pays de l'Adour (UPPA), 2024.

\bibitem{GT98}
D.~Gilbarg and N.~S. Trudinger.
\newblock {\em Elliptic partial differential equations of second order}.
\newblock Classics in Mathematics. Springer-Verlag, Berlin, 2001.
\newblock Reprint of the 1998 edition.

\bibitem{ref-ventcel}
G.~R. Goldstein.
\newblock Derivation and physical interpretation of general boundary conditions.
\newblock {\em Adv. Differential Equations}, 11(4):457--480, 2006.

\bibitem{Grisvard2011}
P.~Grisvard.
\newblock {\em Elliptic problems in nonsmooth domains}, volume~69 of {\em Classics in Applied Mathematics}.
\newblock Society for Industrial and Applied Mathematics (SIAM), Philadelphia, PA, 2011.

\bibitem{haddar-these}
H.~Haddar.
\newblock {\em Mod{\`e}les asymptotiques en ferromagn{\'e}tisme: couches minces et homog{\'e}n{\'e}isation}.
\newblock PhD thesis, Ecole des Ponts ParisTech, 2000.

\bibitem{hansbo2020analysis}
P.~Hansbo, M.~G. Larson, and K.~Larsson.
\newblock Analysis of finite element methods for vector laplacians on surfaces.
\newblock {\em IMA Journal of Numerical Analysis}, 40(3):1652--1701, 2020.

\bibitem{ventcel1}
T.~Kashiwabara, C.~M. Colciago, L.~Ded\`e, and A.~Quarteroni.
\newblock Well-posedness, regularity, and convergence analysis of the finite element approximation of a generalized {R}obin boundary value problem.
\newblock {\em SIAM J. Numer. Anal.}, 53(1):105--126, 2015.

\bibitem{Lenoir1986}
M.~Lenoir.
\newblock Optimal isoparametric finite elements and error estimates for domains involving curved boundaries.
\newblock {\em SIAM J. Numer. Anal.}, 23(3):562--580, 1986.

\bibitem{mora2020apriori}
D.~Mora and G.~Rivera.
\newblock A priori and a posteriori error estimates for a virtual element spectral analysis for the elasticity equations.
\newblock {\em IMA Journal of Numerical Analysis}, 40(1):322--357, 2020.

\bibitem{nedelec}
J.-C. N\'{e}d\'{e}lec.
\newblock Curved finite element methods for the solution of singular integral equations on surfaces in {$R^{3}$}.
\newblock {\em Comput. Methods Appl. Mech. Engrg.}, 8(1):61--80, 1976.

\bibitem{cumin}
C.~Pierre.
\newblock The finite element library {C}umin, curved meshes in numerical simulations.
\newblock {\em repository: https://plmlab.math.cnrs.fr/cpierre1/cumin}, hal-0393713(v1), 2023.

\bibitem{rannacher1980finite}
R.~Rannacher.
\newblock On finite element approximation of general boundary value problems in nonlinear elasticity.
\newblock {\em Calcolo}, 17(2):175--193, 1980.

\bibitem{scott}
L.~R. Scott.
\newblock {\em Finite element techniques for curved boundaries}.
\newblock ProQuest LLC, Ann Arbor, MI, 1973.
\newblock Thesis (Ph.D.)--Massachusetts Institute of Technology.

\bibitem{scott-2}
R.~Scott.
\newblock Interpolated boundary conditions in the finite element method.
\newblock {\em SIAM J. Numer. Anal.}, 12:404--427, 1975.

\bibitem{stewart1998tutorial}
J.~R. Stewart and T.~J. Hughes.
\newblock A tutorial in elementary finite element error analysis: A systematic presentation of a priori and a posteriori error estimates.
\newblock {\em Computer methods in applied mechanics and engineering}, 158(1-2):1--22, 1998.

\bibitem{Ventcel-1956}
A.~D. Ventcel.
\newblock Semigroups of operators that correspond to a generalized differential operator of second order.
\newblock {\em Dokl. Akad. Nauk SSSR (N.S.)}, 111:269--272, 1956.

\bibitem{Ventcel-1959}
A.~D. Ventcel.
\newblock On boundary conditions for multi-dimensional diffusion processes.
\newblock {\em Theor. Probability Appl.}, 4:164--177, 1959.

\bibitem{vial-these}
G.~Vial.
\newblock {\em Analyse asymptotique multi-{\'e}chelle et conditions aux limites approch{\'e}es pour un probl{\`e}me de couche mince dans un domaine {\`a} coin}.
\newblock PhD thesis, Universit{\'e} Rennes 1, 2003.

\end{thebibliography}

\end{document}